\documentclass[10pt]{siamart251216}

\usepackage{geometry,graphicx,amssymb,amsmath,amsfonts,bm,tcolorbox,enumitem,amsbsy}
\usepackage{booktabs,array,multirow,cellspace}
\usepackage[all]{xy}

\usepackage{tikz}
\usetikzlibrary{arrows,shapes}
\usetikzlibrary{decorations.markings,fit}
\usetikzlibrary{calc,3d,decorations.pathmorphing}
\usepackage{tikz-3dplot}
\usepackage{tkz-euclide}
\usepackage{float}
\usetikzlibrary{angles}
\usepackage{pgfplots}
\usepgfplotslibrary{colormaps,patchplots}
\pgfplotsset{compat=1.18}
\usepackage{comment}
\usepackage{tikz}
\usetikzlibrary{arrows,shapes}
\usetikzlibrary{arrows, decorations.markings,fit}
\usetikzlibrary{calc,3d}
\usepackage{tikz-3dplot}
\usepackage{epstopdf}
\usepgfplotslibrary{groupplots}
\usetikzlibrary{matrix, positioning}
\usepackage{xcolor}
\tikzset{
>=stealth',
punkt/.style={
rectangle,
rounded corners,
draw=black, very thick,
text width=6.5em,
minimum height=2em,
text centered},
pil/.style={
->,
thick,
shorten <=2pt,
shorten >=2pt,}
}

\usepackage{subcaption}

\numberwithin{equation}{section}

\allowdisplaybreaks[3]

\newtheorem{example}[theorem]{Example}
\newtheorem{assumption}[theorem]{Assumption}
\newtheorem{remark}[theorem]{Remark}

\usepackage{color}

\newcommand{\R}{\mathbb{R}}
\newcommand{\RRR}{\mathbb{R}}
\newcommand{\bbR}{\mathbb{R}}
\newcommand{\bbN}{\mathbb{N}}

\newcommand{\trialsp}{\calV_{\bfK_{\dive},X}}

\newcommand{\p}{\partial}

\newcommand{\dd}{\mathrm{d}}

\newcommand{\dive}{{\ensuremath\mathop{\mathrm{div}}}}
\newcommand{\curl}{{\ensuremath\mathop{\mathrm{curl}}}}

\newcommand{\supp}{{\mathrm{supp}}}

\newcommand{\bc}{\bm c}

\newcommand{\bff}{\bld{f}}
\newcommand{\bgg}{\bld{g}}

\newcommand{\bu}{\bld{u}}

\newcommand{\bx}{\bld{x}}
\newcommand{\by}{\bld{y}}

\newcommand{\bH}{\bld{H}}

\newcommand{\bfI}{\mathbf{I}}
\newcommand{\bfK}{\mathbf{K}}
\newcommand{\bL}{\bld{L}}

\newcommand{\bW}{\bld{W}}

\newcommand{\bPhi}{\bm \Phi}

\newcommand{\bxi}{\bld{\xi}}

\newcommand{\calB}{\mathcal{B}}

\newcommand{\calD}{\mathcal{D}}

\newcommand{\calF}{\mathcal{F}}
\newcommand{\calH}{\mathcal{H}}
\newcommand{\calI}{\mathcal{I}}

\newcommand{\calN}{\mathcal{N}}
\newcommand{\calO}{\mathcal{O}}

\newcommand{\calV}{\mathcal{V}}

\newcommand{\frE}{\mathfrak{E}}

\newcommand{\Bdel}{\widetilde\calB^{\sigma}_\dive}

\newcommand{\bld}[1]{\boldsymbol{#1}}

\title{Error estimates for vector field interpolation based on generalized matrix-valued kernels\thanks{The first author was supported in part by National Natural Sicence Foundation of China (No. 12571407), Basic Research Program of Jiangsu (No. BK20252037), and a Jiangsu Shuangchuang Team program (No. JSSCTD202449). The third author was supported in part by the General Research Fund (GRF No. 12301520, 12301021, 12300922) of Hong Kong Research Grant Council.}}

\author{
	Zhengjie Sun\thanks{School of Mathematics and Statistics,  Nanjing University of Science and Technology, Nanjing, China (\email{zhengjiesun2020@gmail.com}).}
	\and
	Lishuo Dong\thanks{Corresponding author. School of Mathematics and Statistics,  Nanjing University of Science and Technology, Nanjing, China (\email{dlishuo@163.com}).}
	\and
     Leevan Ling\thanks{Department of Mathematics, Hong Kong Baptist University, Kowloon Tong, Hong Kong (\email{lling@hkbu.edu.hk}).}
}

\headers{Vector field interpolation}{Z. Sun, L. Dong and L. Ling}

\begin{document}
\maketitle
\begin{abstract}
Matrix-valued kernels provide a flexible framework for approximating vector fields from scattered data, especially when structural constraints such as divergence-free or curl-free conditions must be preserved. Classical potential-based constructions enforce these constraints naturally, but they typically require the generating scalar function to possess relatively high smoothness.
We develop an operator-based framework for constructing div-free and curl-free matrix-valued kernels using integral and differential operators, which substantially relaxes the regularity requirements of the potential approach. Using dimension-walking techniques, we show that the resulting native spaces are norm-equivalent to appropriate vector-valued Sobolev spaces. Another main contribution of the paper is a sharp error analysis for the corresponding kernel matrix-valued interpolation problem. We derive direct Sobolev error estimates that allow fractional regularity of the target field, and we establish Bernstein-type inequalities for the associated kernel trial spaces. These results lead to a complete inverse theorem. We also investigate stability by proving lower bounds for the smallest eigenvalues of the interpolation matrices. Numerical experiments are included to verify the theoretical results.
\end{abstract}



\section{Introduction}
Matrix-valued kernels are a central tool in approximation theory and machine learning for reconstructing vector-valued functions from scattered data. Among them, \emph{divergence-free} (div-free) kernels are particularly important because many physical fields satisfy conservation laws. 
For instance, incompressible velocity fields with $\nabla \cdot \mathbf{v} = 0$ and magnetic fields with $\nabla \cdot \mathbf{B} = 0$ \cite{Balsara_2001JCP_divergence,Fuselier_2016ComputFluids_high,Ganesh_2011MCoM_pseudospectral,Guzman_2014MCoM_conforming,Neilan_2021SINUM_divergence}. In contrast to scalar radial basis functions (RBFs), which typically approximate each component separately, matrix-valued kernels can encode cross-component correlations and enforce differential constraints directly, which ensures that the resulting approximant inherits key structural properties of the target field. 

A standard construction of div-free and curl-free matrix-valued kernels applies differential operators to a scalar radial potential \cite{Farrell_2017IMAJNA_multilevel,Fuselier_2008Adv_improved,Fuselier_2008MCoM_sobolev,Keim_2016SINUM_high,Lowitzsch_2005AdvCM_matrix,Narcowich_1994MCoM_generalized,Reisert_2007JMLR_Learning,Schraeder_2011MCoM_high,Wendland_2009SINUM_divergence}. Let $\phi \in C^2([0,\infty))$ be radial and define the associated scalar kernel
$\Phi(\bx,\by)=\phi(\|\bx-\by\|_2)$. The div-free and curl-free matrix-valued kernels $\bPhi_{\dive}$ and $\bPhi_{\curl}$ are then given by
\begin{subequations}\label{eq:potential}
    \begin{align}
        \bPhi_{\dive} &:= (-\Delta \mathbf{I} + \nabla\nabla^\top)\phi, \label{eq:potential_op}\\
        \bPhi_{\curl} &:= -\nabla\nabla^\top \phi, \label{eq:curl_op}
    \end{align}
\end{subequations}
where $\Delta$ denotes the Laplacian, $\nabla$ the gradient, and $\bfI$ the identity matrix in $\R^d$. If $\phi$ is positive definite, then both $\bPhi_{\dive}$ and $\bPhi_{\curl}$ are symmetric positive definite; see \cref{def:SPD}. Moreover, the composite kernel
$\bPhi=-\Delta\phi\,\mathbf{I}$
admits the natural decomposition
$\bPhi=\bPhi_{\dive}+\bPhi_{\curl}$,
see \cite{Amodei_1991JAT_vector,Handscomb_1993NumerAlgor_local,Narcowich_1994MCoM_generalized}. Each column of $\bPhi_{\dive}$ and $\bPhi_{\curl}$ is div-free and curl-free respectively.

While the potential-based construction guarantees that the columns of $\bPhi_{\dive}$ are div-free, with an analogous statement holding for $\bPhi_{\curl}$ in the curl-free case, it imposes stringent regularity requirements on the scalar generator $\phi$. Expanding \eqref{eq:potential_op}, the matrix-valued kernel $\bPhi_{\dive}$ takes the form
\begin{equation} \label{eq:kernel_form_derived}
    \bPhi_{\dive}(\bx,\by)
    =
    \alpha_{\phi}(r)\bfI+\beta_{\phi}(r)(\bx-\by)(\bx-\by)^\top,
    \qquad
    r=\|\bx-\by\|_2,
\end{equation}
where the scalar coefficients are given by
\begin{equation} \label{eq:pot_coeffs}
    \alpha_{\phi}(r) = -\phi''(r) - \frac{d-2}{r}\phi'(r),
    \qquad
    \beta_{\phi}(r) = \frac{1}{r^2}\left(\phi''(r)-\frac{\phi'(r)}{r}\right).
\end{equation}
Similarly, $\bPhi_{\curl}(\bx,\by)$ can be expressed in the same general form \eqref{eq:kernel_form_derived}.
Evidently, this potential-based construction presupposes $\phi$ to be at least twice continuously differentiable. In practice, however, generalized interpolation and collocation methods for partial differential equations typically impose much stronger regularity requirements. For instance, the potential formulation used to solve the Stokes problem in \cite{Wendland_2009SINUM_divergence} requires a compactly supported radial basis function $C^8$, since derivatives up to sixth order must be evaluated. Such stringent smoothness conditions substantially narrow the class of admissible basis functions and make it more difficult to characterize the associated native spaces. Moreover, working with high-order differential operators leads to considerable algebraic overhead.

 Motivated by the representation in \eqref{eq:kernel_form_derived}, we therefore consider the broader family of isotropic matrix-valued kernels of the form
\begin{equation}\label{eq:kernel_form}
    \bfK(\bx,\by)=\alpha(r)\bfI+\beta(r)(\bx-\by)(\bx-\by)^\top,
\end{equation}
where $\alpha,\beta:[0,\infty)\to\R$ are scalar functions whose regularity assumptions will be specified later. Using the operator introduced in \cref{def:IntDeriv_Opt}, the coefficient function $\beta_\phi$
in \eqref{eq:pot_coeffs} can be written succinctly as $\beta_\phi=\mathcal D^2\phi$.
More generally, this observation suggests a unified framework for constructing div-free and curl-free matrix-valued kernels by setting 
$$\beta=\calD^k\phi,~~k\in\bbN,$$
with $k$ controlling the smoothness of the resulting native space. This viewpoint avoids repeatedly differentiating a highly regular potential $\phi$ in closed form, provides greater flexibility in the choice of basis functions, and leads to considerably simpler algebraic expressions in subsequent derivations.

Micheli \cite{Micheli_2013_matrix} proposed a general mechanism for generating matrix-valued kernels from scalar precursors, motivated by applications in shape deformation. In that framework, however, both the explicit form of the resulting kernels and the description of the associated reproducing kernel Hilbert spaces (RKHSs) can become algebraically cumbersome. We show that much of this complexity can be avoided by working with the integral and differential operators $\calI$ and $\calD$, introduced by Wu \cite{Wu_1995Adv_compactly} and Wendland \cite{Wendland_1995Adv_piecewise,Wendland_2004book_scattered}, and subsequently extended by Schaback and Wu \cite{Schaback_1996JCAM_operators} to arbitrary orders $\calI^\nu$ and $\calD^\nu$ with $\nu\in\R$. 
This alternative construction relaxes the stringent smoothness requirements inherent in classical potential formulations. Moreover, combined with dimension-walking techniques, it allows us to identify the induced native spaces with vector-valued Sobolev spaces (up to equivalence of norms).

A second main contribution is a comprehensive error analysis for the corresponding matrix-valued kernel interpolation problem. Although direct (a priori) convergence rates are by now classical \cite{Farrell_2017IMAJNA_multilevel,Fuselier_2008MCoM_sobolev, Sun_2022SISC_kernel,Wendland_2009SINUM_divergence}, we extend the theory to target fields with fractional Sobolev regularity. This refinement is crucial for establishing inverse estimates \cite{Fuselier_2009SINUM_stability,Kuenemund_2019NM_high,Narcowich_2017MCoM_novel,Narcowich_2007FoCM_direct,Narcowich_2006ConsApprox_sobolev,Schaback_2002MCoM_inverse,Sun_2024SISC_high}, which remain significantly less developed—especially for matrix-valued kernels and, more generally, for kernel trial spaces restricted to bounded domains. Recent advances include Wenzel’s sharp inverse theorem for kernel interpolation \cite{Wenzel_2025MCoM_sharp} and the Bernstein and Nikolskii inequalities for kernel-based trial spaces on bounded domains and Riemannian manifolds established by Sun and Ling \cite{Sun_2025_inverse}. Building on these results, we prove Bernstein-type inequalities for matrix-valued kernels and derive a full inverse theorem. In addition, we address stability by establishing lower bounds on the smallest eigenvalues of the associated interpolation matrices.

The remainder of the paper is organized as follows. \Cref{sec:prelim} collects preliminaries and review the operators $\calI$ and $\calD$. \Cref{sec:General_framework} develops the construction of div-free and curl-free kernels and provides necessary and sufficient conditions for their existence. \Cref{sec:error} contains the error analysis, including stability results, direct Sobolev error estimates, and inverse theorems. \Cref{sec:Numer_Examp} reports numerical experiments that confirm the theoretical results.

\section{Preliminaries}
\label{sec:prelim}
\subsection{General notation and function spaces}
Throughout this paper, $\|\cdot\|$ denotes the standard Euclidean norm on $\R^d$. For a matrix $A$, we denote its pseudoinverse by $A^+$. For a matrix-valued function or distribution $\mathbf{G}$, we let $\mathbf{G}^*$ denote its conjugate transpose, i.e., $\mathbf{G}^* := \overline{\mathbf{G}}^\top$. 

We assume that the domain $\Omega \subset \R^d$ is bounded and has a Lipschitz boundary. We say that $\Omega$ has a $C^{k,1}$ boundary if it is locally characterized by the graph of a function that is $k-1$ times continuously differentiable with Lipschitz continuous derivatives of order $k$. Let $X = \{\bx_1, \dots, \bx_N\} \subset \Omega$ be a finite set of distinct points. The \emph{mesh norm} (or fill distance) $h_{X,\Omega}$ and the \emph{separation radius} $q_X$ are defined, respectively, as
\begin{equation*}
    h_{X,\Omega} := \sup_{\bx \in \Omega} \min_{\bx_j \in X} \|\bx - \bx_j\|
    \quad \text{and} \quad
    q_X := \frac{1}{2} \min_{j \neq k} \|\bx_j - \bx_k\|.
\end{equation*}
We say that a sequence $\{X_i\}$ is \emph{quasi-uniform} if all the mesh ratio $\rho_{X_i} := h_{X_i,\Omega}/q_{X_i}$ is bounded uniformly with respect to $i$.

We adopt standard notation for scalar function spaces. For a domain $\Omega \subseteq \R^d$, $C^k(\Omega)$ denotes the space of $k$-times continuously differentiable functions, and $L_p(\Omega)$ denotes the standard Lebesgue space equipped with the norm $\|\cdot\|_{L_p(\Omega)}$. For vector-valued functions $\bff: \Omega \to \R^d$, we employ boldface notation for the corresponding spaces; for instance, $\bff \in \bL_p(\Omega)$ implies that each component of $\bff$ belongs to $L_p(\Omega)$.

For a function $g \in L_1(\R^d)$ or a tempered distribution, we adopt the symmetric normalization for the Fourier transform and its inverse:
\begin{equation*}
    \widehat{g}(\boldsymbol{\xi}) := (2\pi)^{-d/2} \int_{\R^d} g(\bx) e^{-i \bx^\top \boldsymbol{\xi}} \, \dd \bx,
    \quad 
    \check{g}(\bx) := (2\pi)^{-d/2} \int_{\R^d} \widehat{g}(\boldsymbol{\xi}) e^{i \boldsymbol{\xi}^\top \bx} \, \dd \boldsymbol{\xi}.
\end{equation*}
If $\mathbf{G}$ is a matrix-valued function, $\widehat{\mathbf{G}}$ is defined component-wise.

Our error estimates rely on Sobolev spaces. For a non-negative integer $k$ and $1 \le p < \infty$, the Sobolev space $W_p^k(\Omega)$ consists of all functions $f \in L_p(\Omega)$ possessing distributional derivatives $D^\alpha f \in L_p(\Omega)$ for all multi-indices $\alpha$ with $|\alpha| \le k$. The norm is defined as
\begin{equation*}
    \|f\|_{W_p^k(\Omega)} := \Big( \sum_{|\alpha| \le k} \|D^\alpha f\|_{L_p(\Omega)}^p \Big)^{1/p}.
\end{equation*}
For $p=\infty$, the norm is given by $\|f\|_{W_\infty^k(\Omega)} := \max_{|\alpha|\le k} \|D^\alpha f\|_{L_\infty(\Omega)}$. We also consider fractional order Sobolev spaces $W_p^\tau(\Omega)$ for non-integer $\tau$, defined via standard interpolation between integer-order spaces. In the Hilbert space setting ($p=2$), we denote $H^\tau(\Omega) := W_2^\tau(\Omega)$ for any $\tau \ge 0$.

Analogous to the Lebesgue spaces, we write $\bff \in \bW_p^\tau(\Omega)$ if every component $f_j$ of the vector-valued function $\bff$ lies in $W_p^\tau(\Omega)$. We equip this space with the norm
\begin{equation*}
    \|\bff\|_{\bW_p^\tau(\Omega)} := 
    \begin{cases} 
        \displaystyle \Big( \sum_{j=1}^d \|f_j\|_{W_p^\tau(\Omega)}^p \Big)^{1/p} & \text{if } 1 \le p < \infty, \\[1em]
        \displaystyle \max_{1 \le j \le d} \|f_j\|_{W_\infty^\tau(\Omega)} & \text{if } p = \infty.
    \end{cases}
\end{equation*}
With these norms, $\bH^\tau(\Omega) := \bW_2^\tau(\Omega)$ is a Hilbert space. When $\Omega=\R^d$, the inner product is characterized via the Fourier transform:
\begin{equation*}
    \langle \bff, \mathbf{g} \rangle_{\bH^\tau(\R^d)} = \int_{\R^d} (1+\|\bxi\|^2)^\tau  \widehat{\bff}(\bxi)^*\, \widehat{\bgg}(\bxi) \, \dd \bxi.
\end{equation*}

We are particularly interested in subspaces consisting of div-free or curl-free functions. A vector field $\bff:\Omega\to\R^d$ is said to be \emph{div-free} if $\nabla \cdot \bff = 0$. For $\tau \ge 0$, we define the closed subspace
\begin{equation*}
    \bH_{\dive}^\tau(\Omega) := \{ \bff \in \bH^\tau(\Omega) : \nabla \cdot \bff = 0 \}.
\end{equation*}
Similarly, we define spaces for curl-free functions. In three dimensions ($d=3$), a vector field $\bff$ is \emph{curl-free} if $\nabla \times \bff = \mathbf{0}$. In two dimensions ($d=2$), we identify the curl with the scalar operator $\nabla \times \bff := \partial_x f_2 - \partial_y f_1$. Accordingly, for $d \in \{2,3\}$, we define
\begin{equation*}
    \bH_{\text{curl}}^\tau(\Omega) := \{ \bff \in \bH^\tau(\Omega) : \nabla \times \bff = 0 \}.
\end{equation*}

We now turn to the definition of a \emph{native space} for matrix-valued kernels. Analogous to the scalar-valued theory, we characterize these native spaces as reproducing kernel Hilbert spaces.

\begin{definition} \label{def:native_space}
Let $\calH$ be a Hilbert space of vector-valued functions $\bff: \Omega\subseteq \R^d \to \R^d$. A continuous matrix-valued kernel $\bfK: \Omega \times \Omega \to \R^{d \times d}$ is called a \emph{reproducing kernel} for $\calH$ if, for all $\bx \in \Omega$ and $\bc \in \R^d$, the following conditions are satisfied:
\begin{enumerate}[label=(\roman*)]
    \item $\bfK(\cdot, \bx)\bc \in \calH$,
    \item $\bc^\top \bff(\bx) = \langle \bff, \bfK(\cdot, \bx)\bc \rangle_{\calH} \quad \text{for all } \bff \in \calH$.
\end{enumerate}
If such a kernel $\bfK$ exists, we refer to $\calH$ as the \emph{native space} associated with $\bfK$ and denote it by $\calN_{\bfK}$.
\end{definition}

We define positive definite scalar functions and matrix-valued kernels as follows.

\begin{definition}\label{def:SPD}
	A function $\phi:\R^d\to\R$ is positive definite if, for all $N\in\bbN$, all pairwise distinct $\bx_1,\ldots,\bx_N\in\R^d$, and all $\bc\in\R^N\backslash\{\bld{0}\}$, the quadratic form $\sum_{j,k=1}^Nc_jc_k\phi(\bx_j-\bx_k)$ is positive. More generally, a matrix-valued kernel $\bPhi:\R^d\to\R^{d\times d}$ is said to be positive definite if it is even $\bPhi(-\bx)=\bPhi(\bx)$, symmetric $\bPhi(\bx)=\bPhi(\bx)^\top$, and satisfies
	$$\sum_{j,k=1}^N\bc_j^\top \bPhi(\bx_j-\bx_k)\bc_k>0,$$
	for all pairwise distinct $\bx_j\in\R^d$ and all coefficient vectors $\bc_j\in\R^d$ that are not all zero.
\end{definition}

\subsection{Operators on radial functions}

It is a classical result (Bochner's theorem and generalized thereof) that a continuous function is strictly positive definite and radial on $\R^d$ if its $d$-variate Fourier transform is non-negative \cite{Fasshauer_2007book_meshfree,Wendland_2004book_scattered}. For a radial functions $\Phi = \phi(\|\cdot\|_2) \in L_1(\R^d)$, the Fourier transform is itself radial, given by $\widehat{\Phi} = \calF_d \phi(\|\cdot\|_2)$, where the $d$-dimensional radial Fourier operator $\calF_d$ is defined as
\begin{equation*}
    \calF_d \phi(r) = r^{-(d-2)/2} \int_0^\infty \phi(t) t^{d/2} J_{(d-2)/2}(rt) \, \dd t.
\end{equation*}
Here, $J_\nu$ denotes the Bessel function of the first kind of order $\nu$. To analyze the relationship between radial Fourier transforms across different ambient dimensions for compactly supported functions, Schaback and Wu \cite{Schaback_1996JCAM_operators} introduce a certain integral operator and its inverse differential operator.

\begin{definition}\label{def:IntDeriv_Opt}
    The operators $\calI$ and $\calD$ are defined as follows:
    \begin{enumerate}[label=(\roman*)]
        \item Let $\phi:[0,\infty)\to\R$ be a function such that the map $t \mapsto t\phi(t)$ belongs to $L_1[0, \infty)$. The integral operator $\calI$ is defined for $r \ge 0$ by
        \begin{equation} \label{eq:def_I}
            (\calI\phi)(r) := \int_{r}^{\infty} t\phi(t) \, \dd t.
        \end{equation}
        \item Let $\phi \in C^2(\R)$ be an even function. The differential operator $\calD$ is defined for $r \ge 0$ by
        \begin{equation} \label{eq:def_D}
            (\calD\phi)(r) := -\frac{1}{r}\phi'(r).
        \end{equation}
    \end{enumerate}
    In both instances, the resulting functions are understood to be extended to $\R$ as even functions.
\end{definition}

The following lemma establishes the inversion properties of these operators and their dimension-walking effect via Hankel-Bessel recursion \cite{Schaback_1996JCAM_operators,Wendland_1995Adv_piecewise}.

\begin{lemma}
\label{lem:operator_properties}
Let $\phi$ be a continuous function. The operators $\calI$ and $\calD$ satisfy the following properties:
\begin{enumerate}[label=(\roman*)] 
    \item \label{item:inversion}
    If the mapping $t \mapsto t\phi(t)$ belongs to $L_1[0, \infty)$, then $\calD\calI\phi = \phi$. Conversely, if $\phi \in C^2(\R)$ is even and $\phi' \in L_1[0, \infty)$, then $\calI\calD\phi = \phi$.

    \item \label{item:fourier_relation}
    The operators relate the radial Fourier transform in dimension $d$ to dimensions $d-2$ and $d+2$. If $t \mapsto \phi(t)t^{d-1} \in L_1[0, \infty)$ and $d \ge 3$, then
    \begin{equation*}
        \calF_d(\phi) = \calF_{d-2}(\calI\phi).
    \end{equation*}
    Furthermore, if $\phi \in C^2(\R)$ is even and $t \mapsto \phi'(t)t^d \in L_1[0, \infty)$, then
    \begin{equation*}
        \calF_d(\phi) = \calF_{d+2}(\calD\phi).
    \end{equation*}
\end{enumerate}
\end{lemma}

These relations allow us to express the higher-dimensional Fourier transforms of radial functions in terms of lower-dimensional ones, and vice versa.  Since positive definite integrable functions are characterized by a nonnegative Fourier transform, \cref{lem:operator_properties} leads directly to the following results \cite{Wendland_2004book_scattered}.

\begin{lemma}
    Suppose that $\phi$ is continuous. 
        If $t \mapsto \phi(t)t^{d-1} \in L_1[0, \infty)$ and $d \ge 3$, then $\phi$ is positive definite on $\R^d$ if and only if $\calI\phi$ is positive definite on $\R^{d-2}$. 
        Furthermore, if $\phi \in C^2(\R)$ is even and $t \mapsto \phi'(t)t^d \in L_1[0, \infty)$, then $\phi$ is positive definite on $\R^d$ if and only if $\calD\phi$ is positive definite on $\R^{d+2}$.
\end{lemma}

\section{A general framework for constructing matrix-valued kernels}
\label{sec:General_framework}


 \subsection{Matrix-valued kernels} 
While the potential-based differential construction in \eqref{eq:potential} automatically produces div-free or curl-free kernels, it also constrains the coefficient functions $\alpha_\phi$ and $\beta_\phi$ to a particular form. We therefore turn to the more general ansatz \eqref{eq:kernel_form} and derive necessary and sufficient conditions on $\alpha$ and $\beta$ for the resulting matrix-valued kernel to be div-free or curl-free.

\begin{theorem}[div-free condition] \label{thm:div_free}
    Let $\bfK_\dive: \R^d \times \R^d \to \R^{d \times d}$ be an isotropic matrix-valued kernel of the form \eqref{eq:kernel_form}. Then $\bfK_\dive$ satisfies the row-wise div-free condition
    $$
    \nabla_{\bx} \cdot \bfK_\dive(\bx, \by) = \mathbf{0}
    $$
    if and only if the scalar coefficients satisfy the ordinary differential equation
    \begin{equation} \label{eq:ode_condition}
        \alpha'(r) + r^2 \beta'(r) + r(d+1)\beta(r) = 0.
    \end{equation}
\end{theorem}

\begin{proof}
    For the $j$-th component,
    $(\nabla_{\bx} \cdot \bfK_\dive)_j
    =\sum_{i=1}^d \partial_{x_i}(\bfK_\dive)_{ij}$.
    Let $\bu=\bx-\by$. Substituting
    $(\bfK_\dive)_{ij}=\alpha(r)\delta_{ij}+\beta(r)u_i u_j,
    $ and using $
    \partial_{x_i}r=u_i/r,~
    \partial_{x_i}u_j=\delta_{ij}$,
    the derivative of the diagonal term becomes
    $$
    \sum_{i=1}^d \partial_{x_i}\bigl(\alpha(r)\delta_{ij}\bigr)
    =
    \partial_{x_j}\alpha(r)
    =
    \alpha'(r)\frac{u_j}{r}.
    $$
    
    For the rank-one term, the product rule gives
    $$
    \partial_{x_i}\bigl(\beta(r)u_i u_j\bigr)
    =
    \beta'(r)\frac{u_i}{r}u_i u_j
    +
    \beta(r)\delta_{ii}u_j
    +
    \beta(r)u_i\delta_{ij}.
    $$
    Summing over $i$ and using $\sum_{i=1}^d u_i^2=r^2$ and $\sum_{i=1}^d \delta_{ii}=d$, we obtain
    \begin{align*}
        \sum_{i=1}^d \partial_{x_i}\bigl(\beta(r)u_i u_j\bigr)
        &=
        \frac{\beta'(r)u_j}{r}r^2+d\beta(r)u_j+\beta(r)u_j \\
        &=
        \bigl[r\beta'(r)+(d+1)\beta(r)\bigr]u_j.
    \end{align*}
    
    Combining the two contributions yields
    $$
    (\nabla_{\bx}\cdot\bfK_\dive)_j
    =
    \frac{u_j}{r}
    \Bigl[\alpha'(r)+r^2\beta'(r)+r(d+1)\beta(r)\Bigr].
    $$
    Therefore, $\nabla_{\bx}\cdot\bfK_\dive=\mathbf{0}$ for arbitrary $\bu$ if and only if the bracketed term vanishes, which gives \eqref{eq:ode_condition}.
\end{proof}

\begin{remark}
    It is straightforward to check that the coefficients $\alpha_{\phi}$ and $\beta_{\phi}$ in \eqref{eq:pot_coeffs} satisfy \eqref{eq:ode_condition} identically. \cref{thm:div_free}, however, enables a more flexible design principle: rather than starting from a potential $\phi$, one may prescribe the radial function $\beta$ directly and then recover $\alpha$ from \eqref{eq:ode_condition}.
\end{remark}

With the integral operator $\mathcal{I}$ introduced in \cref{def:IntDeriv_Opt}, we obtain the following integral characterization, which enables the direct construction of $\alpha$ from a prescribed $\beta$.

\begin{corollary} \label{cor:alpha_beta_relation}
Under the assumptions of \cref{thm:div_free}, the scalar coefficients $\alpha$ and $\beta$ of a div-free matrix-valued kernel satisfy
\begin{equation} \label{eq:identity}
    \alpha(r) = (d-1)(\mathcal{I}\beta)(r) - r^2\beta(r), \quad r \ge 0,
\end{equation}
provided that $\lim_{r \to \infty} r^2 \beta(r) = 0$.
\end{corollary}

\begin{proof}
Integrating the div-free condition \eqref{eq:ode_condition} over $[r,\infty)$ gives
\begin{equation*}
    \int_r^{\infty} \alpha'(t)\,\mathrm{d}t
    +
    \int_r^{\infty}\left(t^2\beta'(t)+t(d+1)\beta(t)\right)\,\mathrm{d}t
    =0.
\end{equation*}
The condition $\lim_{r \to \infty} r^2 \beta(r) = 0$ implies $\lim_{t \to \infty} \alpha(t) = 0$, thus the first integral becomes
$$
    \int_r^{\infty} \alpha'(t)\,\mathrm{d}t = -\alpha(r).
$$
For the second term, integration by parts yields
\begin{equation*}
    \int_r^{\infty} t^2\beta'(t)\,\mathrm{d}t
    =
    \bigl[t^2\beta(t)\bigr]_r^{\infty}
    -
    \int_r^{\infty}2t\beta(t)\,\mathrm{d}t
    =
    -r^2\beta(r)-2\int_r^{\infty}t\beta(t)\,\mathrm{d}t,
\end{equation*}
where the boundary term at infinity vanishes by the assumption $\lim_{r\to\infty} r^2\beta(r)=0$. Substituting this identity into the integrated equation gives
\begin{equation*}
    -\alpha(r)-r^2\beta(r)+(d-1)\int_r^{\infty} t\beta(t)\,\mathrm{d}t=0.
\end{equation*}
Rearranging, we obtain
$$
    \alpha(r)=(d-1)\int_r^{\infty} t\beta(t)\,\mathrm{d}t-r^2\beta(r).
$$
By \cref{def:IntDeriv_Opt}, we prove \eqref{eq:identity}.
\end{proof}

\begin{remark}
\cref{cor:alpha_beta_relation} yields a constructive kernel-design strategy that bypasses explicit differentiation of a scalar potential. One may prescribe the radial profile $\beta(r)$, for instance, a Gaussian, Mat\'ern, or compactly supported function, compute $\mathcal{I}\beta$ and then recover $\alpha(r)$ from \eqref{eq:identity}. Compared with the potential-based (differential-operator) construction, this route is typically simpler to implement, but it enforces only the div-free constraint. Positive definiteness of the resulting matrix-valued kernel requires an additional spectral condition on $\beta$, namely
$\calD^2\calF_d\beta(\omega)>0$,
this in turn amounts to $\beta$ being positive definite in a sufficiently high dimension. Concrete realizations of this principle are provided in \cref{exam1} and~\cref{exam2}.
\end{remark}

Analogously to div-free kernels, we obtain a simple necessary and sufficient condition for kernels of the form \eqref{eq:kernel_form} to be curl-free.
\begin{theorem}[Curl-free condition] \label{thm:curl_free}
    Let $\bfK_\curl: \R^d \times \R^d \to \R^{d \times d}$ be an isotropic matrix-valued kernel of the form \eqref{eq:kernel_form}, where $\alpha$ and $\beta$ satisfy the assumptions of \cref{lem:operator_properties}. Then $\bfK_\curl$ is curl-free if and only if
    \begin{equation} \label{eq:curl_ode_condition}
        \beta(r) = -\calD\alpha(r),
        \quad \text{or equivalently,} \quad
        \alpha(r) = -\calI\beta(r).
    \end{equation}
\end{theorem}

\begin{proof}
    Let $\bu=\bx-\by$. The curl-free condition means that each column of $\bfK_\curl$ is an irrotational vector field. Equivalently, for every fixed column index $k$, one has
    $$
    \partial_{x_j}(\bfK_\curl)_{ik}
    =
    \partial_{x_i}(\bfK_\curl)_{jk}
    \qquad
    \text{for all } i,j,k.
    $$
    Differentiating the kernel entry $(\bfK_\curl)_{ik} = \alpha(r)\delta_{ik} + \beta(r)u_i u_k$ with respect to $x_j$, and using $\p_{x_j} r = u_j/r$ and $\p_{x_j} u_i = \delta_{ij}$, we obtain
    $$ \p_{x_j} (\bfK_\curl)_{ik} = \alpha'(r)\frac{u_j}{r}\delta_{ik} + \beta'(r)\frac{u_j}{r}u_i u_k + \beta(r)(\delta_{ij}u_k + u_i\delta_{jk}). $$
    Interchanging the indices $i$ and $j$, we likewise obtain
    $$ \p_{x_i} (\bfK_\curl)_{jk} = \alpha'(r)\frac{u_i}{r}\delta_{jk} + \beta'(r)\frac{u_i}{r}u_j u_k + \beta(r)(\delta_{ji}u_k + u_j\delta_{ik}). $$
    Subtracting the two expressions, the terms involving $\beta'(r)$ and $\beta(r)\delta_{ij}u_k$ cancel, and hence
    \begin{align*}
        \p_{x_j} (\bfK_\curl)_{ik} - \p_{x_i} (\bfK_\curl)_{jk} 
        &= \frac{\alpha'(r)}{r} (u_j \delta_{ik} - u_i \delta_{jk}) + \beta(r) (u_i \delta_{jk} - u_j \delta_{ik}) \\
        &= \left[ \frac{\alpha'(r)}{r} - \beta(r) \right] (u_j \delta_{ik} - u_i \delta_{jk}).
    \end{align*}
    For the symmetry condition to hold for arbitrary $\bld{u}$ and indices, the term in the brackets must vanish, yielding \eqref{eq:curl_ode_condition}.
\end{proof}

\subsection{Positive definiteness} 

To carry out the analysis, we pick $\beta = \calD^k \phi$ with $k\in\bbN$ and $\phi$ be some commonly used symmetric positive definite kernels.  The parameter $k$ can, in fact, be extended to arbitrary real values via the fractional operators $\calD^\nu$ ($\nu\in\R$) introduced in \cite{Schaback_1996JCAM_operators}. With the isotropic ansatz \eqref{eq:kernel_form}, the radial coefficients for the div-free kernel $\bfK_\dive$ and the curl-free kernel $\bfK_\curl$ take the form
	\begin{equation}\label{eq:coeff_div}
		\alpha_\dive = (d-1)\calD^{k-1} \phi - r^2\calD^k \phi, \quad \beta_\dive = \calD^k \phi,
	\end{equation}
	and 
	\begin{equation}\label{eq:coeff_curl}
		\alpha_\curl = \calD^{k-1}\phi, \quad \beta_\curl = -\calD^k\phi,
	\end{equation}
	respectively.  The full matrix-valued kernel is then obtained by superposition,
	\begin{equation}\label{eq:combined_kernel}
		\bfK(\bx, \by) = \bfK_\dive(\bx, \by) + \bfK_\curl(\bx, \by), \quad \bx,\by\in\bbR^d.
	\end{equation}
	This choice yields an orthogonal splitting of the native space (\cref{def:native_space}) into div-free and curl-free subspaces, mirroring the classical Helmholtz decomposition for $L_2$ vector fields. As a result, an interpolant built from $\bfK$ automatically decomposes the target field into its solenoidal and potential components.

Because $\bfK(\bx, \by)$ is translation invariant, it is convenient to regard it as a single-variable kernel by setting $\bfK(\bx) := \bfK(\bx, \mathbf{0})$. We analyze its properties via the (matrix-valued) Fourier transform
\begin{equation}\label{eq:Combined_kernel}
	\widehat{\bfK}(\bxi) := (2\pi)^{-d/2}\int_{\bbR^d} \bfK(\bx) e^{-i \bxi^\top \bx}\dd\bx, \quad \bxi \in \bbR^d.
\end{equation}
The div-free constraint $\nabla \cdot \bfK_\dive = \mathbf{0}$ enforces that the corresponding Fourier symbol acts only on directions orthogonal to $\bxi$ (equivalently, it annihilates the component parallel to $\bxi$). Combined with the rotational covariance implied by radial symmetry, this leads to the structural properties stated next.

\begin{lemma}
	\label{lem:general_spectral_relation}
	Let $\phi:[0,\infty)\to\R$ satisfy the assumptions of \cref{lem:operator_properties} such that $\calD^k\phi$ is well defined with $d+4-2k\geq 1$. Let $\bfK$ be defined by \eqref{eq:combined_kernel}. Then, with $\omega=\|\bxi\|$, the Fourier transforms $\widehat \bfK_\dive(\bxi)$ and $\widehat \bfK_\curl(\bxi)$ admit the representations
	\begin{equation}\label{eq:symbol_general}
		\widehat \bfK_\dive(\bxi)=\calF_{d+4-2k}\phi(\omega)\big(\omega^2\bfI-\bxi\bxi^{\top}\big), 
		\quad
		\widehat \bfK_\curl(\bxi)=\calF_{d+4-2k}\phi(\omega)\bxi\bxi^{\top},
	\end{equation}
	where $\calF_{d+4-2k}\phi$ denotes the $(d+4-2k)$-dimensional radial Fourier transform of $\phi$. In particular,
	$\widehat{\bfK}(\bxi)=\omega^2\calF_{d+4-2k}\phi(\omega)\bfI$.
\end{lemma}

\begin{proof}
	Let $g(\omega)=\calF_d(\beta)(\omega)=\calF_d(\calD^k\phi)(\omega)$,
	we first compute the Fourier transform of the term $\beta(\cdot)\bx\bx^\top$. Since $g$ is radial, its Hessian satisfies
	\begin{equation*}
		\nabla_{\bxi}\nabla_{\bxi}^\top g(\omega)
		=
		\frac{g'(\omega)}{\omega}\bfI
		+
		\left(
		g''(\omega)-\frac{g'(\omega)}{\omega}
		\right)\frac{\bxi\bxi^\top}{\omega^2}
		=
		-\calD g(\omega)\bfI+\calD^2 g(\omega)\bxi\bxi^\top.
	\end{equation*}
	Hence,
	\begin{equation}\label{eq:beta_term_op}
		\calF_d\big(\beta(\cdot)\bx\bx^\top\big)(\bxi)
		=
		-\nabla_{\bxi}\nabla_{\bxi}^\top g(\omega)
		=
		\calD g(\omega)\bfI-\calD^2 g(\omega)\bxi\bxi^\top.
	\end{equation}

	Next, we compute the Fourier transforms of the scalar coefficients. Using the identities $$\calF_d(\calI\beta)=\calD\calF_d(\beta)=\calD g,~~\calF_d(r^2\beta)=-\Delta \calF_d(\beta)=-\Delta g,$$
	and
	$\Delta g=\omega^2\calD^2 g-d\calD g$,
	we obtain for the div-free kernel,
	\begin{align*}
		\calF_d(\alpha_\dive)
		&=(d-1)\calF_d(\calI\beta_\dive)-\calF_d(r^2\beta_\dive) \\
		&=(d-1)\calD g-\bigl(d\calD g-\omega^2\calD^2 g\bigr) \\
		&=\omega^2\calD^2 g-\calD g.
	\end{align*}
	Combining this with \eqref{eq:beta_term_op} gives
	\begin{align*}
		\widehat \bfK_\dive(\bxi)
		&=\calF_d(\alpha_\dive)\bfI+\calF_d\big(\beta_\dive(\cdot)\bx\bx^\top\big)(\bxi) \\
		&=\bigl[\omega^2\calD^2 g-\calD g\bigr]\bfI
		+\bigl[\calD g\bfI-\calD^2 g\bxi\bxi^\top\bigr] \\
		&=\calD^2 g\bigl(\omega^2\bfI-\bxi\bxi^\top\bigr).
	\end{align*}

	For the curl-free kernel, $\alpha_\curl=\calI\beta_\curl$, and therefore
	$\calF_d(\alpha_\curl)=\calD g$.
	Hence,
	\begin{align*}
		\widehat \bfK_\curl(\bxi)
		&=\calF_d(\alpha_\curl)\bfI-\calF_d\big(\beta_\curl(\cdot)\bx\bx^\top\big)(\bxi) \\
		&=\calD g\bfI-\bigl[\calD g\bfI-\calD^2 g\bxi\bxi^\top\bigr]
		=\calD^2 g\bxi\bxi^\top.
	\end{align*}

	Finally, by the dimension-walk relation,
	$$
		\calD^2\calF_d(\calD^k\phi)=\calF_{d+4}(\calD^k\phi)=\calF_{d+4-2k}\phi.
	$$
	Substituting this into the preceding identities yields \eqref{eq:symbol_general}. Summing the two expressions in \eqref{eq:symbol_general} gives
	$\widehat{\bfK}(\bxi)=\omega^2\calF_{d+4-2k}\phi(\omega)\bfI$.
\end{proof}

By the preceding theorem, if $\phi$ is positive definite on $\R^{d+4-2k}$, then the associated matrix-valued kernel $\mathbf K$ is positive definite on $\R^d$. The same implication holds for the div-free and curl-free components, $\mathbf K_{\dive}$ and $\mathbf K_{\curl}$, when restricted to their natural subspaces. The classical potential-based construction \cite{Fuselier_2008Adv_improved,Fuselier_2008MCoM_sobolev,Lowitzsch_2005AdvCM_matrix,Wendland_2009SINUM_divergence} in \eqref{eq:potential} is recovered as the special case $k=2$.

Guided by \cref{cor:alpha_beta_relation}, we now describe two particularly simple recipes for div-free kernels; the curl-free counterparts follow by similar arguments. Compared with the scalar potential approach, both constructions reduce the number of derivatives required of the underlying scalar kernel.

\begin{example}[$k=0$]
\label{exam1}
Choose $\beta_{0}(r)=\phi(r)$ in \eqref{eq:identity}. Then
\begin{equation}\label{eq:const_1}
    \alpha_{0}(r)=(d-1)(\mathcal I\phi)(r)-r^2\phi(r).
\end{equation}
If $\phi$ is positive definite on $\R^{d+4}$, then $\mathbf K_{\dive}$ is positive definite on $\R^d$. A notable advantage of this method is that it only requires $\phi$ to be continuous and integrable; no higher-order differentiability is needed.
\end{example}

\begin{example}[$k=1$]
\label{exam2}
Choose $\beta_{1}(r)=\mathcal D\phi(r)$ in \eqref{eq:identity}. Then
\begin{equation}\label{eq:const_2}
    \alpha_{1}(r)=(d-1)\phi(r)-r^2(\mathcal D\phi)(r).
\end{equation}
If $\phi$ is positive definite on $\R^{d+2}$, then $\mathbf K_{\dive}$ is positive definite on $\R^d$.
\end{example}

\subsection{Native space}
We next study the native space $\mathcal{N}_\bfK$ induced by the matrix-valued kernel $\bfK$.
A key aspect of the construction in \cref{cor:alpha_beta_relation} is its smoothing property: the resulting matrix-valued kernel can have higher Sobolev regularity than the scalar kernel from which it is derived. We show, however, that this gain in regularity is accompanied by a trade-off: it requires the underlying scalar kernel to be positive definite in a higher ambient dimension. 
\begin{theorem} \label{thm:native_norm_general}
Let $\phi(r)$ be an even, scalar radial function such that $\calD^k\phi$ is positive definite on $\R^{d+4-2k}$. Under the conditions of  \cref{lem:general_spectral_relation}, the native space
of the matrix-valued kernel $\bfK$ is given by
$$\calN_{\bfK}(\R^d)=\left\{\bff\in\bL_2(\R^d)\cap\bld{C}(\R^d):\|\bff\|_{\calN_{\bfK}}<\infty\right\},$$
equipped with the norm
$$
\|\bff\|_{\calN_{\bfK}}^2 = (2\pi)^{-d/2}\int_{\R^d} \frac{\|\widehat\bff(\bxi)\|^2}{\|\bxi\|^2\, \calF_{d+4-2k}\phi(\|\bxi\|)} \dd\bxi.
$$
Furthermore, the native space admits the orthogonal decomposition $\calN_{\bfK}(\R^d) = \calN_{\bfK_{\dive}}(\R^d) \oplus \calN_{\bfK_{\curl}}(\R^d)$, where the subspaces are explicitly characterized by
\begin{align}
	\calN_{\bfK_{\dive}}(\R^d) &= \left\{ \bff \in \calN_{\bfK}(\R^d) : \nabla \cdot \bff = 0 \right\}, \\
	\calN_{\bfK_{\curl}}(\R^d) &= \left\{ \bff \in \calN_{\bfK}(\R^d) : \nabla \times \bff = \bld{0} \right\}.
\end{align}
\end{theorem}
\begin{proof}
	For translation-invariant matrix kernels, the native space norm is characterized via the Fourier transform as (see, e.g., \cite[Sec.~3.2]{Fuselier_2008Adv_improved})
	\begin{equation}\label{eq:NativespaceNorm}
		\|\bff\|_{\calN_{\bfK}}^2 = (2\pi)^{-d/2}\int_{\R^d} \widehat\bff(\bxi)^* \widehat{\bfK}(\bxi)^{\dagger} \widehat\bff(\bxi) \dd\bxi,
	\end{equation}
	where $\widehat{\bfK}(\bxi)^{\dagger}$ denotes the Moore--Penrose pseudoinverse of the matrix $\widehat{\bfK}(\bxi)$. 
	
	First, consider the combined kernel $\bfK$. Its Fourier transform is a scalar multiple of the identity, $\widehat{\bfK}(\bxi) = \|\bxi\|^2\,\calF_{d+4-2k}\phi(\|\bxi\|) \bfI$. Then the pseudoinverse is given by
	$$
	\widehat{\bfK}(\bxi)^{\dagger}=\frac{1}{\|\bxi\|^2\,\calF_{d+4-2k}\phi(\|\bxi\|)}\bfI.
	$$
	Substituting this into \eqref{eq:NativespaceNorm} yields the native space norm.
	
	Next, we characterize the subspaces. The Fourier transform of the div-free kernel factors as
	$$
	\widehat{\bfK}_\dive(\bxi) = \|\bxi\|^2\,\calF_{d+4-2k}\phi(\|\bxi\|) \mathbf{P}_{\bxi}, \quad \text{where} \quad \mathbf{P}_{\bxi} = \bfI - \frac{\bxi\bxi^\top}{\|\bxi\|^2}
	$$
	is the orthogonal projector onto the subspace orthogonal to $\bxi$. Since the scalar factor is non-zero for $\|\bxi\| > 0$, the pseudoinverse is $\widehat{\bfK}_\dive(\bxi)^{\dagger} = (\|\bxi\|^2\,\calF_{d+4-2k}\phi(\|\bxi\|))^{-1} \mathbf{P}_{\bxi}$. For the native space integral \eqref{eq:NativespaceNorm} to be finite, $\widehat\bff(\bxi)$ must lie in the range of $\widehat{\bfK}_\dive(\bxi)$ almost everywhere. This implies $\mathbf{P}_{\bxi} \widehat\bff(\bxi) = \widehat\bff(\bxi)$, or equivalently $\bxi^\top \widehat\bff(\bxi) = 0$, which corresponds to the physical condition $\nabla \cdot \bff = 0$. 
	Under this condition, the quadratic form simplifies to
	$$\widehat\bff(\bxi)^* \widehat{\bfK}_\dive(\bxi)^{\dagger} \widehat\bff(\bxi) = \frac{\widehat\bff(\bxi)^* \mathbf{P}_{\bxi} \widehat\bff(\bxi)}{\|\bxi\|^2 \,\calF_{d+4-2k}\phi(\|\bxi\|)} = \frac{\|\widehat\bff(\bxi)\|^2}{\|\bxi\|^2 \,\calF_{d+4-2k}\phi(\|\bxi\|)}.$$
	An analogous argument applies to $\bfK_{\curl}$ using the complementary projector $\bfI - \mathbf{P}_{\bxi}$. Since $\mathbf{P}_{\bxi}$ and $\bfI - \mathbf{P}_{\bxi}$ are orthogonal projectors summing to the identity, the decomposition $\calN_{\bfK} = \calN_{\bfK_{\dive}} \oplus \calN_{\bfK_{\curl}}$ follows immediately.
\end{proof}


\smallskip
\section{Stability and error estimates of matrix-valued kernel interpolation}
\label{sec:error}

Under standard decay assumptions on the Fourier transform, the native space $\mathcal N_{\mathbf K}(\R^d)$ can be identified with a subspace of a Sobolev space, with equivalence of norms. To make this precise, let $\psi_m:[0,\infty)\to\R$ be radial and suppose that its $d$-dimensional Fourier transform satisfies
\begin{equation} \label{eq:psi_hat_decay}
	\calF_d\psi_m(\omega) \asymp (1+\omega^{2})^{-m}, \quad m > d/2.
\end{equation}
Let $\bfK$ be the matrix-valued kernel of the form \eqref{eq:combined_kernel} with 
	$\beta(r) = \calD^k\psi_m(r), ~~k\in\bbN$. 
 We also introduce the Sobolev subspace $\widetilde{\bH}^{\tau}(\R^d)$ as
\begin{equation} \label{eq:Sub_divspace}
	\widetilde{\bH}^{\tau}(\R^d) := \left\{\bff \in \bH^\tau(\R^d) :  \|\bff\|_{\widetilde{\bH}^{\tau}} < \infty \right\},
\end{equation}
equipped with the norm
\begin{equation} \label{eq:H_norm}
	\|\bff\|_{\widetilde{\bH}^{\tau}}^2 := (2\pi)^{-d/2} \int_{\R^d} \frac{\|\widehat{\bff}(\bxi)\|^2}{\|\bxi\|^{2}} (1 + \|\bxi\|^2)^{\tau+1}  \, \dd\bxi.
\end{equation}

\begin{theorem} \label{thm:native_space_Fourier_decay}
	Let $\bfK$ be the matrix-valued kernel of the form \eqref{eq:combined_kernel} with $\beta=\calD^k\psi_m$, where $\psi_m$ satisfies \eqref{eq:psi_hat_decay}. If $\psi_m$ is positive definite on $\R^{d+4-2k}$, then $\bfK$ is positive definite on $\R^d$. Moreover, $\calN_{\bfK}(\R^d)\cong\widetilde\bH^{m+1-k}(\R^{d})$ with equivalence of norms:
	\begin{equation} \label{eq:native_norm_equiv}
		\|\bff\|_{\calN_{\bfK}} \asymp \|\bff\|_{\widetilde\bH^{m + 1-k}(\R^{d})}, \quad \bff \in \calN_{\bfK}.
	\end{equation}
\end{theorem}

\begin{proof}
	By \cref{lem:general_spectral_relation}, we have $\widehat{\bfK}(\bxi)=\omega^2\calF_{d+4-2k}\psi_m(\omega)\bfI$, where $\omega=\|\bxi\|$. 
	To determine the norm equivalence, we analyze the asymptotic decay of $\widehat{\bfK}$. Since $\calF_d\psi_m(\omega) \asymp (1+\omega^2)^{-m}$, the decay rate in the shifted dimension $d' = d+4-2k$ is adjusted by half the difference in dimensions. Specifically,
	\begin{equation*}
		\calF_{d'}\psi_m(\omega) \asymp (1+\omega^2)^{-(m + \lfloor\frac{d'-d}{2}\rfloor)} = (1+\omega^2)^{-(m + 2 - k)}.
	\end{equation*}
	By \cref{thm:native_norm_general}, the native space norm is characterized by
	\begin{align*}
		\|\bff\|^{2}_{\calN_{\bfK}}
		&= (2\pi)^{-d/2}\int_{\R^{d}} \frac{\|\widehat{\bff}(\bxi)\|^{2}}{\|\bxi\|^2\calF_{d+4-2k}\psi_m(\|\bxi\|)} \, \dd\bxi \\
		&\asymp (2\pi)^{-d/2}\int_{\R^{d}} \frac{\|\widehat{\bff}(\bxi)\|^{2}}{\|\bxi\|^2}(1+\|\bxi\|^2)^{m+2-k} \, \dd\bxi.
	\end{align*}
	Comparing this to \eqref{eq:H_norm} with $\tau = m+1-k$, which completes the proof.
\end{proof}

\begin{remark}
\cref{thm:native_space_Fourier_decay} shows that, for a fixed generator $\psi_m$, the native space associated with $\bfK$ has Sobolev order $m+1-k$. Thus, smaller values of $k$ lead to higher regularity of the native space. At the same time, the positive definiteness requirement is imposed on $\psi_m$ in the higher-dimensional space $\R^{d+4-2k}$. For kernels that are positive definite in every dimension, such as the Gaussian and the inverse multiquadric, this additional requirement is immaterial.
\end{remark}

In the subsequent sections, we study stability and error estimates for div-free matrix-valued kernel interpolation. The corresponding results in the curl-free setting follow similarly.
\subsection{Stability}
Let $X = \{\bx_1, \dots, \bx_N\} \subset \R^d$ be a set of pairwise distinct nodes, and let $\{\bff_j = \bff(\bx_j)\}_{j=1}^N \subset \R^d$ denote data sampled from a target div-free field $\bff$. We define the div-free trial space as
\begin{equation}\label{eq:trialsp}
    \trialsp := \Big\{\sum_{j=1}^N \bfK_\dive(\cdot - \bx_j) \bc_j : \bx_j \in X, \, \bc_j \in \R^d\Big\}.
\end{equation}
Let $I_X\bff \in \trialsp$ be the unique interpolant determined by
\begin{equation} \label{eq:interp_cond}
    (I_X\bff)(\bx_k) = \sum_{j=1}^N \bfK_\dive(\bx_k - \bx_j) \bc_j = \bff_k, \quad k=1, \dots, N.
\end{equation}
Equivalently, \eqref{eq:interp_cond} can be written as the block linear system
\begin{equation*}
    A_{\bfK_\dive,X} \bc = \bff|_X,
\end{equation*}
where $\bc = [\bc_1^\top, \dots, \bc_N^\top]^\top$ and $\bff|_X = [\bff_1^\top, \dots, \bff_N^\top]^\top$ are vectors in $\R^{dN}$. The interpolation matrix $A_{\bfK_\dive,X} \in \R^{dN \times dN}$ consists of  $N \times N$ blocks of size $d \times d$, with $(k,j)$-th block $(A_{\bfK_\dive,X})_{k,j} = \bfK_\dive(\bx_k - \bx_j)$. If $\bfK_\dive$ is positive definite, then $A_{\bfK_\dive,X}$ is symmetric and positive definite, and the interpolation problem is well posed.

Lower bounds for the smallest eigenvalue (and hence stability estimates) for div-free kernel interpolation were derived by Fuselier \cite[Thm.~6]{Fuselier_2008Adv_improved}. That analysis treats kernels of the form $\Phi = (-\Delta \bfI + \nabla\nabla^\top)\phi$, generated from a scalar potential $\phi$, and expresses the eigenvalue bound in terms of the Fourier transform $\widehat{\phi}$. To apply this result in our setting, we note from \cref{lem:general_spectral_relation} that $\bfK_\dive$ can be written in potential form:
\begin{equation}\label{eq:dive_ker_diff}
    \bfK_\dive = (-\Delta \bfI + \nabla\nabla^\top) \calI^2 \beta_{\dive} = (-\Delta \bfI + \nabla\nabla^\top) \calI^{2-k}\psi_m.
\end{equation}
Thus, upon identifying the scalar potential as $\phi = \calI^{2-k}\psi_m$, Fuselier’s framework yields the following stability result.

\begin{theorem}
\label{thm:min_eig}
 Suppose the assumptions of \cref{thm:native_space_Fourier_decay} hold. Let $\bfK_\dive$ be the div-free kernel defined as in \eqref{eq:coeff_div} with $\beta_\dive = \calD^k\psi_m$. Define the auxiliary function
\begin{equation*}
    M(\delta) := \inf_{\|\bxi\| \leq \delta} \widehat{\calI^{2-k}\psi_m}(\|\bxi\|).
\end{equation*}
Then, there exists a constant $\tilde{c} > 0$, independent of $\psi_m$ and $X$, such that for all $\delta \geq \tilde{c}/q_X$, the smallest eigenvalue of the interpolation matrix satisfies
\begin{equation}\label{eq:minimum_eigenvalue}
    \lambda_{\min}(A_{\bfK_\dive,X}) \ge \frac{\pi}{(4\pi)^2\Gamma((d+2)/2)} \left(\frac{\delta^2}{16\pi}\right)^{(d+2)/2} M(\delta).
\end{equation}
Moreover, if the generator $\psi_m$ satisfies the decay condition \eqref{eq:psi_hat_decay}, then
\begin{equation*}
    \lambda_{\min}(A_{\bfK_\dive,X}) \ge c_d \, q_X^{2m - d - 2k + 2},
\end{equation*}
where $c_d$ is a positive constant depending only on the dimension $d$.
\end{theorem}

\begin{proof}
The estimate \eqref{eq:minimum_eigenvalue} follows immediately from \cite[Thm.~6]{Fuselier_2008Adv_improved} together with the potential representation \eqref{eq:dive_ker_diff}. For the second claim, the decay assumption \eqref{eq:psi_hat_decay} implies that there exists a constant $\tilde{c}_1 > 0$ such that
\begin{equation*}
    \widehat{\calI^{2-k}\psi_m}(\|\bxi\|) \ge \tilde{c}_1 (1 + \|\bxi\|^2)^{-(m + 2 - k)}.
\end{equation*}
Hence, for $\delta \geq 1$, the auxiliary function satisfies
\begin{equation*}
    M(\delta) \ge \tilde{c}_1 \delta^{-2(m + 2 - k)}.
\end{equation*}
Now choose $\delta = \tilde{c}/q_X$. For $q_X$ sufficiently small this ensures $\delta \ge 1$, and inserting the above lower bound for $M(\delta)$ into \eqref{eq:minimum_eigenvalue} gives
\begin{equation*}
    \lambda_{\min}(A_{\bfK_\dive,X}) \ge C \, q_X^{-(d+2)} q_X^{2(m + 2 - k)} = C \, q_X^{2m - d - 2k + 2},
\end{equation*}
where $C>0$ denotes a generic constant independent of $X$. This completes the proof.
\end{proof}

\subsection{Direct estimates}
Throughout this section, we impose the following assumptions on the domain, the discretization, and the kernel.

\begin{assumption}\label{Assump_domain}
Let $\Omega \subset \R^d$ be a bounded, simply connected domain with boundary of class $\mathcal C^{\lceil m\rceil,1}$. Let $X\subset \Omega$ be a quasi-uniform set of centers with fill distance $h_X$, separation distance $q_X$, and mesh ratio $\rho_X:=h_X/q_X$. Assume further that the trial space $\trialsp$ in \eqref{eq:trialsp} is generated by the div-free kernel $\mathbf K_{\dive}$ of the form \eqref{eq:kernel_form}, with coefficient $\beta_{\dive}=\mathcal D\psi_m$, where the scalar kernel $\psi_m$ satisfies the decay condition \eqref{eq:psi_hat_decay}.
\end{assumption}

We restrict attention to the choice $\beta_{\dive}=\mathcal D\psi_m$. In this case, the native space $\mathcal N_{\mathbf K_{\dive}}$ is norm-equivalent to $\widetilde{\bH}^{m}_{\dive}(\R^d)$. The analysis for other kernel constructions is analogous. Throughout this section, $C$ denotes a generic positive constant independent of the discretization parameters $h_X$ and $q_X$, whose value may vary from line to line.

The existence of a fractional extension operator preserving the div-free constraint was established in \cite[Prop.~3.8]{Wendland_2009SINUM_divergence}.

\begin{lemma}\label{lem:Extension}
Let $\Omega$ satisfy \cref{Assump_domain}, and let $m\in\R^{+}$. Then there exists a continuous extension operator
\[
\mathfrak E_{\dive}:\bH^{m}_{\dive}(\Omega)\to \widetilde{\bH}^{m}_{\dive}(\R^d)
\]
such that $\mathfrak E_{\dive}\bff|_{\Omega}=\bff$ for all $\bff\in \bH^{m}_{\dive}(\Omega)$.
\end{lemma}

Using this result, we derive error estimates for RBF approximations of div-free vector fields in the native space. The following theorem extends \cite[Thm.~5]{Fuselier_2008MCoM_sobolev} to the case of noninteger smoothness $m$.

\begin{theorem}\label{thm:WithinNativeSp}
Suppose \cref{Assump_domain} holds. Let $q \in [1, \infty]$ and let $m > d/2$ be a real number. If $\bff \in \bH^m(\Omega)$ is div-free, then
\[
\|\bff - I_X \bff\|_{W_q^\mu(\Omega)} \leq C h_{X,\Omega}^{m - \mu - d(1/2 - 1/q)_+} \|\bff\|_{\bH^m(\Omega)},
\]
for all $\mu$ satisfying $0 \leq \mu \leq m-d(1/2 - 1/q)_+$.
\end{theorem}

\begin{proof}
Since the error function $\bff - I_X \bff$ vanishes on the set $X$, an application of \cite[Corol.~4.7]{Wendland_2009SINUM_divergence} yields
\[
\|\bff - I_X \bff\|_{W_q^\mu(\Omega)} \leq C h_{X,\Omega}^{m - \mu - d(1/2 - 1/q)_+} \|\bff - I_X \bff\|_{\bH^m(\Omega)}.
\]
Recall that the native space is equivalent to $\widetilde\bH_{\dive}^m(\R^d)$. By \cref{lem:Extension}, we can continuously extend $\bff$ to $\widetilde\bH_{\dive}^m(\R^d)$ using the operator $\frE_{\dive}$. We then use the best approximation property of the interpolant in the native space, to obtain
\begin{align*}
\|\bff - I_X \bff\|_{\bH^m(\Omega)} 
&\le C \|\frE_{\dive} \bff - I_X \bff\|_{\widetilde\bH^m(\R^d)} \\
&\le C \|\frE_{\dive} \bff\|_{\widetilde\bH^m(\R^d)} \\
&\leq C \|\bff\|_{\bH^m(\Omega)}.
\end{align*}
The proof is completed by combining the preceding estimates.
\end{proof}

 Using the generalized sampling inequalities established in \cite{Arcangeli_2007NumerMath_extension,LeGia_2006JAT_continuous} and adapting the techniques from \cite{avesani_2025arXiv_sobolev,Fuselier_2008MCoM_sobolev}, we derive the following error estimates for functions outside the native space. This result is also a slight refinement of \cite[Thm.~6]{Fuselier_2008MCoM_sobolev} by weakening the assumptions.
\begin{theorem}
    Suppose \cref{Assump_domain} holds. Let $q \in [1, \infty]$ and let $\tau$ be a real number satisfying $m \ge \tau \ge \lfloor \tau \rfloor > d/2$. For any div-free vector field $\bff \in \bH^\tau(\Omega)$, the error estimate
    \begin{equation*}
        \|\bff - I_X\bff\|_{\bW_q^{\mu}(\Omega)} \leq C h_{X}^{\tau - \mu - d(1/2 - 1/q)_+} \rho_{X}^{m - \tau} \|\bff\|_{\bH^\tau(\Omega)}
    \end{equation*}
    holds for all $0 \leq \mu < \lfloor \tau \rfloor - d/2$.
\end{theorem}

\begin{proof}
    The proof adapts the arguments presented in \cite[Thm.~3.2]{avesani_2025arXiv_sobolev} and \cite[Thm.~5]{Fuselier_2008MCoM_sobolev}. Since the residual $\bff - I_X\bff$ vanishes on the set $X$, we may apply the sampling inequalities established in \cite[Thm.~2.1]{LeGia_2006JAT_continuous} and \cite[Thm.~2.12]{Narcowich_2005MCoM_sobolev}. Specifically, for any integer $\ell$ such that $0 \le \ell \le \lfloor \tau \rfloor - d/2$, we have
    \begin{equation*}
        \|\bff - I_X\bff\|_{\bW_q^{\ell}(\Omega)} \leq C h_{X}^{\tau - \ell - d(1/2 - 1/q)_+} \|\bff - I_X\bff\|_{\bH^\tau(\Omega)}.
    \end{equation*}
    To bound the $\bH^\tau$-norm on the right-hand side, we employ the techniques from the proof of \cite[Thm.~6]{Fuselier_2008MCoM_sobolev}, which yield the estimate
    \begin{equation*}
        \|\bff - I_X\bff\|_{\bH^\tau(\Omega)} \leq C \rho_{X}^{m - \tau} \|\bff\|_{\bH^\tau(\Omega)}.
    \end{equation*}
    Combining these estimates results in
    \begin{equation*}
        \|\bff - I_X\bff\|_{\bW_q^{\ell}(\Omega)} \leq C h_{X}^{\tau - \ell - d(1/2 - 1/q)_+} \rho_{X}^{m - \tau} \|\bff\|_{\bH^\tau(\Omega)}.
    \end{equation*}
    The result for arbitrary real $\mu$ in the specified range follows by interpolating between the cases $\ell = 0$ and $\ell = \lfloor \tau \rfloor - d/2$.
\end{proof}

\subsection{Inverse estimates}

While the direct estimates in the previous section quantify the approximation power of the trial space, a full stability and convergence analysis also requires inverse estimates, i.e., bounds of stronger norms of discrete trial functions in terms of weaker norms. Such inequalities inevitably involve negative powers of the separation radius $q_X$. In this section, we derive inverse estimates for the trial space $\trialsp$. Our approach exploits the connection between discrete trial functions and a suitable space of band-limited extensions.

To this end, we introduce the band-limited space
\begin{equation*}
    \widetilde\calB^\sigma := \left\{\bff \in \bL_2(\R^d) : \supp(\widehat\bff) \subseteq B(0,\sigma) \quad \text{and} \quad \int_{\R^{d}} \frac{\|\widehat{\bff}(\bxi)\|^{2}}{\|\bxi\|^2} \dd\bxi < \infty \right\}.
\end{equation*}
The next lemma summarizes the properties of $\widetilde\calB^\sigma$ needed below, in particular those related to interpolation and approximation. Although the proof relies on ideas from \cite{Fuselier_2008MCoM_sobolev,Wendland_2009SINUM_divergence}, we provide it here for completeness, since the specific formulation required in our analysis does not appear explicitly in those references.

\begin{lemma}\label{lem:band_estimates}
    Suppose \cref{Assump_domain} holds. Let $m$ and $\tau$ satisfy $m \ge \tau > d/2$. For any div-free vector field $\bu \in \bH^{m}(\Omega)\subseteq\bH^\tau(\Omega)$, there exists a band-limited function
    \begin{equation*}
        \bff_{\sigma,\bu,\tau} \in \calB^\sigma_{\dive} := \left\{\bff \in \widetilde\calB^\sigma : \bxi^\top\widehat\bff(\bxi)=0 \text{ for almost all } \bxi \right\}
    \end{equation*}
    with bandwidth $\sigma = \mathcal{O}(q_X^{-1})$ such that $\bff_{\sigma}:=\bff_{\sigma,\bu,\tau}$ interpolates $\frE_{\dive}\bu$ on $X$, that is,
    \begin{equation}\label{eq:band_1}
        \bff_{\sigma}|_X = (\frE_{\dive}\bu)|_X,
    \end{equation}
    and satisfies
    \begin{subequations}
    \begin{align}
        \|\bff_{\sigma}\|_{\widetilde\bH^{\tau}(\R^d)} &\leq C \|\bu\|_{\bH^{\tau}(\Omega)}, \label{eq:band_2}\\
        \|\bu - \bff_{\sigma}\|_{\bH^{\tau}(\Omega)} &\leq C q_X^{m-\tau} \|\bu\|_{\bH^{m}(\Omega)}. \label{eq:band_3}
    \end{align}
    \end{subequations}
    Moreover, for any $\mu \in [0,m]$ and any $\bff_\sigma \in \calB^\sigma_{\dive}$, the Bernstein inequality
    \begin{equation}\label{eq:band_4}
        \|\bff_\sigma\|_{\widetilde\bH^{m}(\R^d)} \leq C \sigma^{m-\mu} \|\bff_\sigma\|_{\widetilde\bH^{\mu}(\R^d)}
    \end{equation}
    holds.
\end{lemma}

\begin{proof}
    We first prove \eqref{eq:band_2}. By the construction of the band-limited interpolant (see, for example, \cite{Fuselier_2008MCoM_sobolev}), together with the continuity of the extension operator $\frE_{\dive}$ from \cref{lem:Extension}, we obtain
    \begin{equation*}
        \begin{aligned}
        \|\bff_{\sigma}\|_{\widetilde\bH^{\tau}(\R^d)}
        &\le \|\bff_{\sigma} - \frE_{\dive} \bu\|_{\widetilde\bH^{\tau}(\R^d)} + \|\frE_{\dive} \bu\|_{\widetilde\bH^{\tau}(\R^d)} \\
        &\le C_1 \|\frE_{\dive} \bu\|_{\widetilde\bH^{\tau}(\R^d)} + C_2 \|\bu\|_{\bH^{\tau}(\Omega)} \\
        &\le C \|\bu\|_{\bH^{\tau}(\Omega)}.
        \end{aligned}
    \end{equation*}

    We next prove \eqref{eq:band_3}. Since $\bff_\sigma$ interpolates $\frE_{\dive}\bu$ on $X$, we have $I_X\bu = I_X\bff_\sigma$. By applying the triangle inequality, combined with the error estimates for RBF interpolation from \cref{thm:WithinNativeSp} and the band-limited approximation bounds from \cite[Lem.~5]{Fuselier_2008MCoM_sobolev}, we get
    \begin{equation*}
        \begin{aligned}
        \|\bu - \bff_{\sigma}\|_{\bH^{\tau}(\Omega)} &\le \|\bu - I_X\bu\|_{\bH^{\tau}(\Omega)} + \|I_X \bff_{\sigma} - \bff_{\sigma}\|_{\bH^{\tau}(\Omega)} \\
        &\le C h_{X}^{m-\tau} \|\bu\|_{\bH^{m}(\Omega)} + C h_{X}^{m-\tau} \|\bff_{\sigma}\|_{\widetilde\bH^{m}(\R^d)}.
        \end{aligned}
    \end{equation*}
    Using the stability estimate \eqref{eq:band_2} to bound the second term, and noting that $h_X \le C q_X$ for quasi-uniform points, the result follows.
    
    Finally, we verify the Bernstein inequality \eqref{eq:band_4}. For any $\bxi \in B(0, \sigma)$, the inequality $(1+\|\bxi\|^2)^{m+1} \le C\sigma^{2(m-\mu)}(1+\|\bxi\|^2)^{\mu+1}$ holds. Consequently,
    \begin{align*}
        \|\bff_\sigma\|_{\widetilde\bH^{m}(\R^d)}^2 &= (2\pi)^{-d/2} \int_{\|\bxi\|\le \sigma} \frac{\|\widehat{\bff}_\sigma(\bxi)\|^{2}}{\|\bxi\|^2}(1+\|\bxi\|^2)^{m+1} \dd\bxi \\
        &\le C (2\pi)^{-d/2} \sigma^{2(m-\mu)} \int_{\|\bxi\|\le \sigma} \frac{\|\widehat{\bff}_\sigma(\bxi)\|^{2}}{\|\bxi\|^2}(1+\|\bxi\|^2)^{\mu+1} \dd\bxi \\
        &= C \sigma^{2(m-\mu)} \|\bff_\sigma\|_{\widetilde\bH^{\mu}(\R^d)}^2.
    \end{align*}
    Taking the square root completes the proof.
\end{proof}

The following theorem provides a Bernstein-type inequality relating the $\bH^m$-norm to $\bH^\tau$ with $d/2 < \tau \leq m$. 
\begin{theorem}\label{thm:first_inv1}
    Let $m$ and $\tau$ be real numbers satisfying $d/2 < \tau \leq m$. There exists a constant $C$, independent of the point set $X$, such that for all trial functions $\bu \in \trialsp$, the following estimate holds:
    \begin{equation}\label{eq:inv_est_1}
        \|\bu\|_{\bH^m(\Omega)} \leq C q_X^{-(m-\tau)} \|\bu\|_{\bH^\tau(\Omega)}.
    \end{equation}
\end{theorem}

\begin{proof}
    By \cref{lem:band_estimates}, for any $\bu \in \trialsp\subseteq\widetilde\bH^m(\R^d)\subseteq\widetilde\bH^\tau(\R^d)$, there exists a band-limited function $\bff_\sigma:=\bff_{\sigma,\bu,\tau} \in \Bdel$ with bandwidth $\sigma = \mathcal{O}(q_X^{-1})$ such that properties \eqref{eq:band_1}--\eqref{eq:band_4} hold. We first bound the $\bH^m$-norm of the discrete function. Using the triangle inequality and the identity $\bu = I_X \bu = I_X \bff_\sigma$, we derive:
    \begin{equation*}
        \begin{aligned}
            \|\bu\|_{\bH^m(\Omega)} &\leq \|\bu - \bff_\sigma\|_{\bH^m(\Omega)} + \|\bff_\sigma\|_{\bH^m(\Omega)} \\
            &= \|I_X \bff_\sigma - \bff_\sigma\|_{\bH^m(\Omega)} + \|\bff_\sigma\|_{\bH^m(\Omega)} \\
            &\leq C \|\bff_\sigma\|_{\widetilde\bH^m(\R^d)}.
        \end{aligned}
    \end{equation*}
    Here, the final inequality follows from the approximation property of the interpolant $I_X$.
    
    To further bound the right-hand side, we apply the Bernstein inequality \eqref{eq:band_4} and the stability estimate \eqref{eq:band_2}. This yields:
    \begin{equation*}
        \|\bff_\sigma\|_{\widetilde\bH^m(\R^d)} \leq C q_X^{-(m-\tau)} \|\bff_\sigma\|_{\widetilde\bH^\tau(\R^d)} \leq C q_X^{-(m-\tau)} \|\bu\|_{\bH^\tau(\Omega)}.
    \end{equation*}
    Combining these estimates gives the desired result.
\end{proof}

Building on \cref{thm:first_inv1}, we can derive a general inverse estimate relating the Sobolev norm of any order $\mu \in [0, m]$ to the $L_2$-norm. This result is particularly useful for stability and error analysis in time-dependent problems.

\begin{theorem}
    Under \cref{Assump_domain}, for any $\mu \in [0, m]$, there exists a constant $C$ independent of $q_X$ such that the Bernstein-type inequality
    \begin{equation}\label{eq:BernIneq}
        \|\bu\|_{\bH^\mu(\Omega)} \leq C q_X^{-\mu} \|\bu\|_{\bL_2(\Omega)}
    \end{equation}
    holds for all point sets $X \subset \Omega$ and all trial functions $\bu \in \trialsp$.
\end{theorem}

\begin{proof}
    We first establish the result for a fixed intermediate regularity $\tau \in (d/2, m)$. By the Gagliardo--Nirenberg interpolation inequality, there exists a constant $C$ such that
    \begin{equation*}
        \|\bu\|_{\bH^\tau(\Omega)} \leq C \|\bu\|_{\bL_2(\Omega)}^{1-\tau/m} \|\bu\|_{\bH^m(\Omega)}^{\tau/m}.
    \end{equation*}
    Substituting the inverse estimate from \cref{thm:first_inv1} into the term $\|\bu\|_{\bH^m(\Omega)}$, we have
    \begin{equation*}
        \begin{aligned}
            \|\bu\|_{\bH^\tau(\Omega)} &\leq C \|\bu\|_{\bL_2(\Omega)}^{1-\tau/m} \left( q_X^{-(m-\tau)} \|\bu\|_{\bH^\tau(\Omega)} \right)^{\tau/m} \\
            &= C q_X^{-\tau(1-\tau/m)} \|\bu\|_{\bL_2(\Omega)}^{1-\tau/m} \|\bu\|_{\bH^\tau(\Omega)}^{\tau/m}.
        \end{aligned}
    \end{equation*}
    Dividing both sides by $\|\bu\|_{\bH^\tau(\Omega)}^{\tau/m}$ (assuming $\|\bu\|_{\bH^\tau(\Omega)} \neq 0$) and raising the resulting expression to the power $1/(1-\tau/m)$ yields
    \begin{equation}\label{eq:L2_inv_tau}
        \|\bu\|_{\bH^\tau(\Omega)} \leq C q_X^{-\tau} \|\bu\|_{\bL_2(\Omega)}.
    \end{equation}
    
    Now consider an arbitrary $\mu \in [0, m)$. If $\mu < \tau$, we apply the Gagliardo--Nirenberg inequality again, interpolating between $L_2(\Omega)$ and $H^\tau(\Omega)$:
    \begin{equation*}
        \begin{aligned}
            \|\bu\|_{\bH^\mu(\Omega)} &\leq C \|\bu\|_{\bL_2(\Omega)}^{1-\mu/\tau} \|\bu\|_{\bH^\tau(\Omega)}^{\mu/\tau} \\
            &\leq C \|\bu\|_{\bL_2(\Omega)}^{1-\mu/\tau} \left( q_X^{-\tau} \|\bu\|_{\bL_2(\Omega)} \right)^{\mu/\tau} \\
            &= C q_X^{-\mu} \|\bu\|_{\bL_2(\Omega)}.
        \end{aligned}
    \end{equation*}
    Finally, for the case $\mu = m$, we select a specific $\tau_0$ such that $d/2 < \tau_0 < m$. Combining \cref{thm:first_inv1} with the estimate \eqref{eq:L2_inv_tau} derived for $\tau_0$, we obtain
    \begin{equation*}
        \|\bu\|_{\bH^m(\Omega)} \leq C q_X^{-(m-\tau_0)} \|\bu\|_{\bH^{\tau_0}(\Omega)} \leq C q_X^{-m+\tau_0} q_X^{-\tau_0} \|\bu\|_{\bL_2(\Omega)} = C q_X^{-m} \|\bu\|_{\bL_2(\Omega)}.
    \end{equation*}
    This completes the proof.
\end{proof}

\begin{theorem}
Under \cref{Assump_domain}, let $(X_n)_{n\in\mathbb{N}}\subset\Omega$ be a nested sequence of quasi-uniform point sets exhibiting geometric decay; specifically, assume there exist constants $c_0', c_0 > 0$ and $a \in (0,1)$ such that $c_0' a^{n} \le q_{X_n} \le h_{X_n} \le c_0 a^{n}$. Let $\bff \in \bL_2(\Omega)$ and assume there exists a sequence of approximants $(\bff_{X_n})_{n\in\mathbb{N}} \subset \calN_\bfK$, with $\bff_{X_n} \in \calV_{\bfK,X_n}$, satisfying
\begin{equation}\label{eq:Inv_cond}
    \|\bff -\bff_{X_n}\|_{\bL_2(\Omega)} \leq c_{\bff} h_{X_n, \Omega}^\tau
\end{equation}
for some $c_{\bff} > 0$ and $\tau \in (0, m]$. Then $\bff \in \bH^{\tau'}(\Omega)$ for all $\tau' \in (0, \tau)$.
\end{theorem}

\begin{proof}
Following the approach in \cite{Narcowich_2007FoCM_direct,Schaback_2002MCoM_inverse,Wenzel_2025MCoM_sharp}, we consider the difference between successive approximants, $\bff_{X_{n+1}} - \bff_{X_n}$. The nestedness assumption ensures that $\bff_{X_{n+1}} - \bff_{X_n} \in \calV_{\bfK,X_{n+1}}$. Applying the Bernstein inequality \eqref{eq:BernIneq} with $\tau' < \tau$ yields
\[
    \|\bff_{X_{n+1}} - \bff_{X_n}\|_{\bH^{\tau'}(\Omega)} \leq C q_{X_{n+1}}^{-\tau'} \|\bff_{X_{n+1}} - \bff_{X_n}\|_{\bL_2(\Omega)}.
\]
Here, we have used the property that $X_n \subset X_{n+1}$, which implies $q_{X_n \cup X_{n+1}} = q_{X_{n+1}}$. By the triangle inequality and the assumed decay rates, we obtain
\begin{align*}
    \|\bff_{X_{n+1}} - \bff_{X_n}\|_{\bH^{\tau'}(\Omega)} &\leq C q_{X_{n+1}}^{-\tau'} \|(\bff_{X_{n+1}} - \bff) - (\bff_{X_n} - \bff)\|_{\bL_2(\Omega)} \\
    &\leq C q_{X_{n+1}}^{-\tau'} \left( c_{\bff} h_{X_{n+1}, \Omega}^\tau + c_{\bff} h_{X_n, \Omega}^\tau \right) \\
    &\leq 2 c_{\bff} C q_{X_{n+1}}^{-\tau'} h_{X_n, \Omega}^\tau \\
    &\leq 2 c_{\bff} C c_0'^{-\tau'} c_0^\tau a^{-\tau'} a^{n(\tau - \tau')}.
\end{align*}
Standard arguments demonstrate that $(\bff_{X_n})_{n\in\mathbb{N}}$ is a Cauchy sequence in $\bH^{\tau'}(\Omega)$ and therefore converges to $\bff \in \bH^{\tau'}(\Omega)$ by the completeness of the space; see \cite{avesani_2025arXiv_sobolev,Narcowich_2007FoCM_direct,Wenzel_2025MCoM_sharp}.
\end{proof}

	\section{Numerical examples}
\label{sec:Numer_Examp}

    \definecolor{myred}{RGB}{195,0,0}
	\definecolor{myblue}{RGB}{0,90,170}
	\definecolor{mygreen}{RGB}{0,140,0}
	\definecolor{myreddark}{RGB}{140,0,0}
	\definecolor{myredlight}{RGB}{255,100,100}
	\definecolor{mybluedark}{RGB}{0,50,120}
	\definecolor{mybluelight}{RGB}{100,180,255}
	\definecolor{mygreendark}{RGB}{0,90,0}
	\definecolor{mygreenlight}{RGB}{100,220,100}
	\definecolor{mypurple}{RGB}{120,0,120}
	\definecolor{myorange}{RGB}{230,120,0}
	\definecolor{mycyan}{RGB}{0,150,150}
	\definecolor{myyellow}{RGB}{210,180,0}
	\definecolor{mybrown}{RGB}{150,100,50}
	\definecolor{mygray}{RGB}{120,120,120}

In this section, we present numerical experiments designed to validate the theoretical findings on convergence rates and stability for our generalized matrix-valued kernels. We compare the classical potential-based construction in \eqref{eq:kernel_form_derived}--\eqref{eq:pot_coeffs} with the more general isotropic ansatz
\[
\mathbf{K}(\boldsymbol{x},\boldsymbol{y})=\alpha(r)\mathbf{I}+\beta(r)(\boldsymbol{x}-\boldsymbol{y})(\boldsymbol{x}-\boldsymbol{y})^\top,
\quad r=\|\boldsymbol{x}-\boldsymbol{y}\|.
\]
To simplify the presentation, we parameterize all kernels by the choice of $\beta(r)$. Once $\beta$ is fixed, $\alpha(r)$ is uniquely determined by the div-free (or curl-free) constraint, so that each choice yields an admissible kernel. We consider three cases:
$$\mathbf{K}^{(0)}: \beta(r)=\phi(r);\quad \mathbf{K}^{(1)}: \beta(r)=\calD\phi(r);\quad \mathbf{K}^{(2)}: \beta(r)=\calD^2\phi(r).$$
Here, $\mathbf{K}^{(0)}$ and $\mathbf{K}^{(1)}$ are proposed in this work (see \cref{exam1} and \cref{exam2}), whereas $\mathbf{K}^{(2)}$ corresponds to the classical potential-based construction commonly used in the literature.

 We consider two standard families of scalar radial kernels that satisfy the decay condition \eqref{eq:psi_hat_decay}: Mat\'ern (MA) kernels and Wendland (WE) kernels. The Mat\'ern family is given by
$$\phi_{\nu}(r) =\frac{2^{1-(\nu-d/2)}}{\Gamma(\nu-d/2)} (\varepsilon r)^{\nu-d/2}K_{\nu-d/2}(\varepsilon r),$$ where  $K_\mu$ denotes the modified Bessel function of the second kind of order $\mu$, and $\varepsilon>0$ is a shape parameter. Mat\'ern kernels are strictly positive definite on $\RRR^d$, and their smoothness is controlled by the parameter $\nu$. For example, choosing $\nu=\frac{d+3}{2}$ yields $\phi_\nu\in C^{2}(\RRR^d)$, which $\nu=\frac{d+5}{2}$ gives $\phi_\nu\in C^{4}(\RRR^d)$. As a compactly supported alternative, we also employ Wendland functions $\phi_{d,\ell}:[0,\infty)\rightarrow \R$ from \cite{Wendland_1995Adv_piecewise,Wendland_2004book_scattered}. These kernels are strictly positive definite on $\RRR^d$ and satisfy $\phi_{d,\ell}\in C^{2\ell}(\R^d)$. Moreover, their native spaces are Sobolev spaces $H^{\ell+\frac{d}{2}+\frac{1}{2}}(\R^d)$.
In our experiments, we use the specific Wendland kernel
$$\phi_{5,2}(r) = (1-\varepsilon r)_{+}^7\big(16(\varepsilon r)^2+7\varepsilon r+1\big),$$
which is strictly positive definite on $\RRR^d$ for dimensions $d\leq 5$ and belongs to $C^4(\RRR^d)$.

	\subsection{Convergence test}
	
	In this example, we test the convergence and numerical stability of the three interpolation methods for a vector field on $\RRR^2$.  Following \cite{Wendland_2009SINUM_divergence}, we consider the target field
	\begin{equation} \label{eq:target_u}
		\boldsymbol{u}(\boldsymbol{x}) = \begin{pmatrix} 20x_1x_2^3+3x_1^2-3x_2^2 \\ 5x_1^4 - 5x_2^4-6x_1x_2 \end{pmatrix}.
	\end{equation}
	This field splits naturally into a div-free part $\boldsymbol{u}_1$ and a curl-free part $\boldsymbol{u}_2$:
	\begin{equation}
		\boldsymbol{u}_1(\boldsymbol{x}) = \begin{pmatrix} 20x_1x_2^3 \\ 5x_1^4 - 5x_2^4 \end{pmatrix}, \quad \boldsymbol{u}_2(\boldsymbol{x}) = \begin{pmatrix} 3x_1^2-3x_2^2 \\ -6x_1x_2 \end{pmatrix}.
	\end{equation}
	
	We first test uniform node sets $X \subset \Omega$ with fill distances $h_X \in \{0.4, 0.2, 0.1, 0.05, 0.025\}$. Errors are measured in the discrete $\ell_2$-norm on a dense uniform evaluation grid $Y$ with mesh size $h_Y=0.01$, namely
	\[
\mathcal{E}(X) := \|\boldsymbol{u}- I_X \boldsymbol{u}\|_{\bm{\ell}_2(Y)}.
	\]
	Interpolation is performed using the Mat\'ern kernel $\phi_{7/2}$ and the Wendland kernel $\phi_{5,2}$. In all runs the shape parameters are held fixed (no tuning is performed).  For the Mat\'ern kernel, we take $\varepsilon_1 = 2.7$, $\varepsilon_2 = 3$ and $\varepsilon_3= 3.5$; for the Wendland kernel we use $\varepsilon_1 = 0.37$, $\varepsilon_2 = 0.5$ and $\varepsilon_3= 0.6$.

	\cref{fig.Err_Field1} displays the interpolation errors for the full field $\bu$ and for its orthogonal components using three methods with both the Mat\'{e}rn and Wendland kernels. \cref{tab:conv_rates} summarizes the corresponding average convergence orders, which are calculated via a least-squares linear fit of the log-log error profiles from the coarsest to the finest point sets. The average convergence orders in \cref{tab:conv_rates} agree with the direct estimates proved in \Cref{sec:error}: $\mathbf{K}^{(0)}$ achieves the fastest asymptotic decay at approximately $\calO(h^{4.5})$, followed by $\mathbf{K}^{(1)}$ at $\calO(h^{3.5})$ and $\mathbf{K}^{(2)}$ at $\calO(h^{2.5})$. Furthermore, to assess robustness with respect to nonuniform sampling, we repeat the experiment on Halton points. With the same shape parameters, \cref{fig.Err_Field2} shows essentially the same convergence behavior, which indicates that the convergences persist for scattered data.  
	
	Next, we investigate the numerical stability of the three methods by calculating the minimum eigenvalues of the interpolation matrices using the Mat\'{e}rn kernel $\phi_{7/2}$ (see \cref{fig.cond1}). We observe that the decay rates with respect to the separation distances $q_X$ are $\mathcal{O}(q_X^7)$ for $\mathbf{K}^{(0)}$, $\mathcal{O}(q_X^5)$ for $\mathbf{K}^{(1)}$ and $\mathcal{O}(q_X^3)$ for $\mathbf{K}^{(2)}$. These rates perfectly corroborate the theoretical estimate $q_X^{2m-d-2k+2}$ provided in \cref{thm:min_eig} for $m=7/2$ and $d=2$. This phenomenon is also supported by Schaback’s uncertainty relation, which asserts that the spectral norm of the inverse of an interpolation matrix is inversely proportional to the approximation order \cite{Schaback_1995Adv_error}. Furthermore, we plot the minimum eigenvalues for $\mathbf{K}^{(0)}$ using the Wendland kernels $\phi_{3,2}$, $\phi_{5,2}$ and $\phi_{7,2}$ with the same shape parameter $\varepsilon=0.37$. Note that for uniform nodes at $h = 0.025$, the minimum eigenvalues associated with $\phi_{3,2}$ and $\phi_{5,2}$ become negative; hence, these data points are excluded from the figure. Additionally, for the Halton point set of $N = 1600$, the eigenvalues corresponding to $\phi_{3,2}$ also become negative. This behavior is entirely consistent with our theoretical framework, which dictates that $\mathbf{K}^{(0)}$ necessitates a kernel that is strictly positive definite in $\R^{d+4}$.

	\begin{figure}
		\centering
		\begin{tikzpicture}
			\begin{groupplot}[
				group style={
					group size=3 by 1, 
					horizontal sep=0pt, 
					vertical sep=0pt,
				},
				width=4cm, height=5cm, 
				scale only axis,           
				xmode=log, ymode=log,
				xmin=0.006, xmax=2,
				ymin=1e-7, ymax=10,
				grid=both,
				major grid style={line width=0.2pt, draw=gray!40, dashed},
				minor grid style={line width=0.1pt, draw=gray!15, dotted},
				minor x tick num=3, minor y tick num=3,
				ticklabel style={font=\tiny},
				xlabel={$h$},
				xlabel style={font=\small, yshift=6pt},
				legend style={
					at={(1,0)}, 
					anchor=south east,
					font=\tiny,
                    nodes={inner ysep=1pt},
					cells={anchor=west},
					legend columns=1,
					inner sep=1pt,
					outer sep=1pt,
					draw=none,
					fill opacity=0.8,
					text opacity=1
				},
				legend image post style={mark size=1.8pt},
				cycle list={
					{color={myred}, mark=triangle*, line width=0.5pt,mark size=1.5pt},
					{color={myblue}, mark=square*, line width=0.5pt,mark size=1.5pt},
					{color={mygreen}, mark=diamond*, line width=0.5pt,mark size=1.5pt},
					{color={myred}, mark=triangle*, line width=0.5pt,mark size=1.5pt,dashed},
					{color={myblue}, mark=square*, line width=0.5pt,mark size=1.5pt,dashed},
					{color={mygreen}, mark=diamond*, line width=0.5pt,mark size=1.5pt,dashed}
				},
				]
				
				\nextgroupplot[
				ylabel={$\ell_2$ error}, 
				title={combined},
				title style={font=\large, yshift=-3pt},
				]
				
				\addplot table[x index=0, y index=2, col sep=space] {field1comma.txt};
				\addlegendentry{MA, $\mathbf{K}^{(0)}$}
				
				\addplot table[x index=0, y index=3, col sep=space] {field1comma.txt};
				\addlegendentry{MA, $\mathbf{K}^{(1)}$}
                
				\addplot table[x index=0, y index=1, col sep=space] {field1comma.txt};
				\addlegendentry{MA, $\mathbf{K}^{(2)}$}
				
				\addplot table[x index=0, y index=2, col sep=space] {field1comwe52.txt};
				\addlegendentry{WE, $\mathbf{K}^{(0)}$}
				
				\addplot table[x index=0, y index=3, col sep=space] {field1comwe52.txt};
				\addlegendentry{WE, $\mathbf{K}^{(1)}$}
                
				\addplot table[x index=0, y index=1, col sep=space] {field1comwe52.txt};
				\addlegendentry{WE, $\mathbf{K}^{(2)}$}
				
				\addplot[
				domain=0.03:0.4,
				samples=2,
				color=black,
				line width=0.5pt,  %
				dashdotted,
				forget plot,
				] {20 * x^4.5};
				\node[anchor=south west, font=\tiny] at (axis cs:0.2,0.008) {$h^{4.5}$};
				\addplot[
				domain=0.025:0.3,
				samples=2,
				color=black,
				line width=0.5pt,  %
				dashdotted,
				forget plot,
				] {45 * x^2.5};
				\node[anchor=south west, font=\tiny] at (axis cs:0.04,0.1) {$h^{2.5}$};

				\nextgroupplot[
				title={div-free},
				title style={font=\large, yshift=-3pt},
				yticklabels={},                  
				ylabel={},
				axis y line*=right, 
				]

				\addplot table[x index=0, y index=2, col sep=space] {field1divma.txt};
				\addlegendentry{MA, $\mathbf{K}^{(0)}$}
				
				\addplot table[x index=0, y index=3, col sep=space] {field1divma.txt};
				\addlegendentry{MA, $\mathbf{K}^{(1)}$}
                
				\addplot table[x index=0, y index=1, col sep=space] {field1divma.txt};
				\addlegendentry{MA, $\mathbf{K}^{(2)}$}
				
				\addplot table[x index=0, y index=2, col sep=space] {field1divwe52.txt};
				\addlegendentry{WE, $\mathbf{K}^{(0)}$}
				
				\addplot table[x index=0, y index=3, col sep=space] {field1divwe52.txt};
				\addlegendentry{WE, $\mathbf{K}^{(1)}$}
                
				\addplot table[x index=0, y index=1, col sep=space] {field1divwe52.txt};
				\addlegendentry{WE, $\mathbf{K}^{(2)}$}
				
				\addplot[
				domain=0.03:0.4,
				samples=2,
				color=black,
				line width=0.5pt,  %
				dashdotted,
				forget plot,
				] {20 * x^4.5};
				\node[anchor=south west, font=\tiny] at (axis cs:0.2,0.008) {$h^{4.5}$};
				\addplot[
				domain=0.025:0.3,
				samples=2,
				color=black,
				line width=0.5pt,  %
				dashdotted,
				forget plot,
				] {45 * x^2.5};
				\node[anchor=south west, font=\tiny] at (axis cs:0.04,0.1) {$h^{2.5}$};
				
				%

				\nextgroupplot[
				title={curl-free},
				title style={font=\large, yshift=-3pt},
				yticklabels={},                 
				ylabel={},
				axis y line*=right, 
				]
				
				\addplot table[x index=0, y index=2, col sep=space] {field1curlma.txt};
				\addlegendentry{MA, $\mathbf{K}^{(0)}$}
				
				\addplot table[x index=0, y index=3, col sep=space] {field1curlma.txt};
				\addlegendentry{MA, $\mathbf{K}^{(1)}$}
                
				\addplot table[x index=0, y index=1, col sep=space] {field1curlma.txt};
				\addlegendentry{MA, $\mathbf{K}^{(2)}$}
				
				\addplot table[x index=0, y index=2, col sep=space] {field1curlwe52.txt};
				\addlegendentry{WE, $\mathbf{K}^{(0)}$}
				
				\addplot table[x index=0, y index=3, col sep=space] {field1curlwe52.txt};
				\addlegendentry{WE, $\mathbf{K}^{(1)}$}
                
				\addplot table[x index=0, y index=1, col sep=space] {field1curlwe52.txt};
				\addlegendentry{WE, $\mathbf{K}^{(2)}$}
				
				\addplot[
				domain=0.03:0.4,
				samples=2,
				color=black,
				line width=0.5pt,  %
				dashdotted,
				forget plot,
				] {5 * x^4.5};
				\node[anchor=south west, font=\tiny] at (axis cs:0.2,0.0025) {$h^{4.5}$};
				\addplot[
				domain=0.025:0.3,
				samples=2,
				color=black,
				line width=0.5pt,  %
				dashdotted,
				forget plot,
				] {10 * x^2.5};
				\node[anchor=south west, font=\tiny] at (axis cs:0.06,0.05) {$h^{2.5}$};
				
				%
				
			\end{groupplot}
		\end{tikzpicture}
		
		\captionsetup{font=normalsize}
		\caption{Discrete $\ell_2$ errors for the vector field $\boldsymbol{u}$ and its div-free and curl-free components $\boldsymbol{u}_1$ and $\boldsymbol{u}_2$, computed with matrix-valued kernel interpolation using Mat\'{e}rn and Wendland kernels on \textbf{uniform nodes} with $h=\{0.4,0.2,0.1,0.05,0.025\}$.}
		\label{fig.Err_Field1}
	\end{figure}
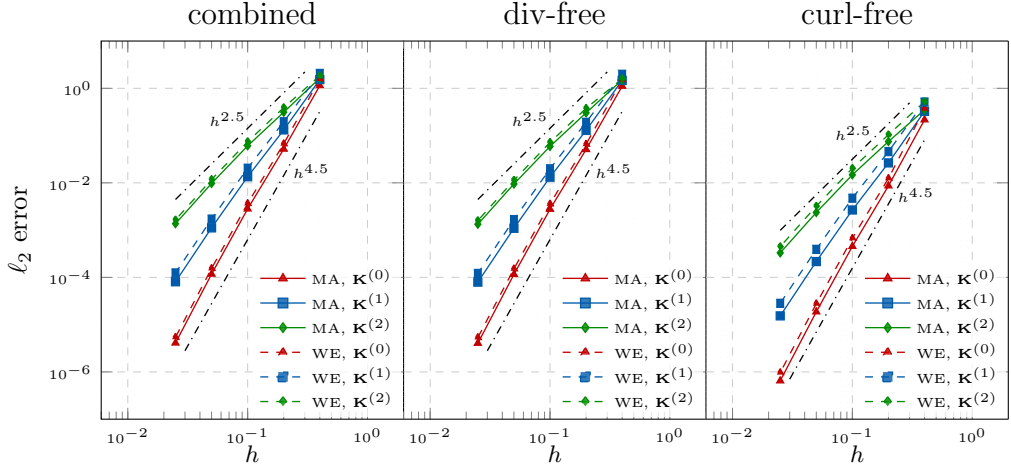

	\begin{table}
		\centering
		\caption{Averaged convergence orders of three cases ($\mathbf{K}^{(0)}$, $\mathbf{K}^{(1)}$, $\mathbf{K}^{(2)}$) for approximating $\bld{u}$, $\bld{u}_1$ and $\bld{u}_2$ on uniform nodes using Mat\'{e}rn and Wendland kernels.}
		\label{tab:conv_rates}
		\setlength{\tabcolsep}{4mm}
		\begin{tabular}{lccccc}
			\hline
			Kernel & Case  & $\boldsymbol{u}$ & $\boldsymbol{u}_1$ & $\boldsymbol{u}_2$& Theoretical rate \\ \hline
			& $\mathbf{K}^{(0)}$  & 4.50 & 4.49 & 4.55 & 4.5\\
			Mat\'{e}rn & $\mathbf{K}^{(1)}$  &3.53 & 3.52 & 3.57& 3.5 \\
            & $\mathbf{K}^{(2)}$  & 2.55 & 2.53 & 2.51 & 2.5\\ \hline
			& $\mathbf{K}^{(0)}$  & 4.51 & 4.50 & 4.59 & 4.5\\
			Wendland & $\mathbf{K}^{(1)}$  & 3.49 & 3.48 & 3.52& 3.5 \\
            & $\mathbf{K}^{(2)}$  & 2.53 & 2.50 & 2.54 & 2.5\\ \hline
		\end{tabular}
	\end{table}

\begin{figure}
	\centering
	\begin{tikzpicture}
		\begin{groupplot}[
			group style={
				group size=3 by 1, 
				horizontal sep=0pt, 
				vertical sep=0pt,
			},
			width=4cm, height=5cm, 
			scale only axis,           
			xmode=log, ymode=log,
			xmin=0.006, xmax=2,
			ymin=1e-6, ymax=10,
			grid=both,
			major grid style={line width=0.2pt, draw=gray!40, dashed},
			minor grid style={line width=0.1pt, draw=gray!15, dotted},
			minor x tick num=3, minor y tick num=3,
			ticklabel style={font=\tiny},
			xlabel={$h$},
			xlabel style={font=\small, yshift=6pt},
			legend style={
				at={(1,0)}, 
				anchor=south east,
				font=\tiny,
                nodes={inner ysep=1pt},
				cells={anchor=west},
				legend columns=1,
				inner sep=1pt,
				outer sep=1pt,
				draw=none,
				fill opacity=0.8,
				text opacity=1
			},
			legend image post style={mark size=1.8pt},
			cycle list={
				{color={myred}, mark=triangle*, line width=0.5pt,mark size=1.5pt},
				{color={myblue}, mark=square*, line width=0.5pt,mark size=1.5pt},
				{color={mygreen}, mark=diamond*, line width=0.5pt,mark size=1.5pt},
				{color={myred}, mark=triangle*, line width=0.5pt,mark size=1.5pt,dashed},
				{color={myblue}, mark=square*, line width=0.5pt,mark size=1.5pt,dashed},
				{color={mygreen}, mark=diamond*, line width=0.5pt,mark size=1.5pt,dashed}
			},
			]
			
			\nextgroupplot[
			ylabel={$\ell_2$ error}, 
			title={combined},
			title style={font=\large, yshift=-3pt},
			]
			
			\addplot table[x index=0, y index=2, col sep=space] {haltoncomma.txt};
			\addlegendentry{MA, $\mathbf{K}^{(0)}$}
			
			\addplot table[x index=0, y index=3, col sep=space] {haltoncomma.txt};
			\addlegendentry{MA, $\mathbf{K}^{(1)}$}
            
			\addplot table[x index=0, y index=1, col sep=space] {haltoncomma.txt};
			\addlegendentry{MA, $\mathbf{K}^{(2)}$}
			
			\addplot table[x index=0, y index=2, col sep=space] {haltoncomwe52.txt};
			\addlegendentry{WE, $\mathbf{K}^{(0)}$}
			
			\addplot table[x index=0, y index=3, col sep=space] {haltoncomwe52.txt};
			\addlegendentry{WE, $\mathbf{K}^{(1)}$}
            
			\addplot table[x index=0, y index=1, col sep=space] {haltoncomwe52.txt};
			\addlegendentry{WE, $\mathbf{K}^{(2)}$}
			
			\addplot[
			domain=0.04:0.3,
			samples=2,
			color=black,
			line width=0.5pt,  %
			dashdotted,
			forget plot,
			] {70 * x^4.5};
			\node[anchor=south west, font=\tiny] at (axis cs:0.2,0.018) {$h^{4.5}$};
			\addplot[
			domain=0.025:0.2,
			samples=2,
			color=black,
			line width=0.5pt,  %
			dashdotted,
			forget plot,
			] {200 * x^2.5};
			\node[anchor=south west, font=\tiny] at (axis cs:0.04,0.7) {$h^{2.5}$};

			\nextgroupplot[
			title={div-free},
			title style={font=\large, yshift=-3pt},
			yticklabels={},                  
			ylabel={},
			axis y line*=right, 
			]

			\addplot table[x index=0, y index=2, col sep=space] {haltondivma.txt};
			\addlegendentry{MA, $\mathbf{K}^{(0)}$}
			
			\addplot table[x index=0, y index=3, col sep=space] {haltondivma.txt};
			\addlegendentry{MA, $\mathbf{K}^{(1)}$}
            
			\addplot table[x index=0, y index=1, col sep=space] {haltondivma.txt};
			\addlegendentry{MA, $\mathbf{K}^{(2)}$}
			
			\addplot table[x index=0, y index=2, col sep=space] {haltondivwe52.txt};
			\addlegendentry{WE, $\mathbf{K}^{(0)}$}
			
			\addplot table[x index=0, y index=3, col sep=space] {haltondivwe52.txt};
			\addlegendentry{WE, $\mathbf{K}^{(1)}$}
            
			\addplot table[x index=0, y index=1, col sep=space] {haltondivwe52.txt};
			\addlegendentry{WE, $\mathbf{K}^{(2)}$}
			
			\addplot[
			domain=0.04:0.3,
			samples=2,
			color=black,
			line width=0.5pt,  %
			dashdotted,
			forget plot,
			] {70 * x^4.5};
			\node[anchor=south west, font=\tiny] at (axis cs:0.2,0.018) {$h^{4.5}$};
			\addplot[
			domain=0.025:0.2,
			samples=2,
			color=black,
			line width=0.5pt,  %
			dashdotted,
			forget plot,
			] {200 * x^2.5};
			\node[anchor=south west, font=\tiny] at (axis cs:0.04,0.7) {$h^{2.5}$};
			
			%

			\nextgroupplot[
			title={curl-free},
			title style={font=\large, yshift=-3pt},
			yticklabels={},                 
			ylabel={},
			axis y line*=right, 
			]
			
			\addplot table[x index=0, y index=2, col sep=space] {haltoncurlma.txt};
			\addlegendentry{MA, $\mathbf{K}^{(0)}$}
			
			\addplot table[x index=0, y index=3, col sep=space] {haltoncurlma.txt};
			\addlegendentry{MA, $\mathbf{K}^{(1)}$}
            
			\addplot table[x index=0, y index=1, col sep=space] {haltoncurlma.txt};
			\addlegendentry{MA, $\mathbf{K}^{(2)}$}
			
			\addplot table[x index=0, y index=2, col sep=space] {haltoncurlwe52.txt};
			\addlegendentry{WE, $\mathbf{K}^{(0)}$}
			
			\addplot table[x index=0, y index=3, col sep=space] {haltoncurlwe52.txt};
			\addlegendentry{WE, $\mathbf{K}^{(1)}$}
            
			\addplot table[x index=0, y index=1, col sep=space] {haltoncurlwe52.txt};
			\addlegendentry{WE, $\mathbf{K}^{(2)}$}
			
			\addplot[
			domain=0.03:0.3,
			samples=2,
			color=black,
			line width=0.5pt,  %
			dashdotted,
			forget plot,
			] {20 * x^4.5};
			\node[anchor=south west, font=\tiny] at (axis cs:0.2,0.008) {$h^{4.5}$};
			\addplot[
			domain=0.025:0.2,
			samples=2,
			color=black,
			line width=0.5pt,  %
			dashdotted,
			forget plot,
			] {70 * x^2.5};
			\node[anchor=south west, font=\tiny] at (axis cs:0.06,0.3) {$h^{2.5}$};
			
			%
			
		\end{groupplot}
	\end{tikzpicture}
	
	\captionsetup{font=normalsize}
	\caption{Discrete $\ell_2$ errors for the vector field $\boldsymbol{u}$ and its div-free and curl-free components $\boldsymbol{u}_1$ and $\boldsymbol{u}_2$, computed with matrix-valued kernel interpolation using Mat\'{e}rn and Wendland kernels on \textbf{Halton nodes} with $N=\{6,25,100,400,1600\}$ (with $h\sim 1/\sqrt{N}$).}
	\label{fig.Err_Field2}
\end{figure}
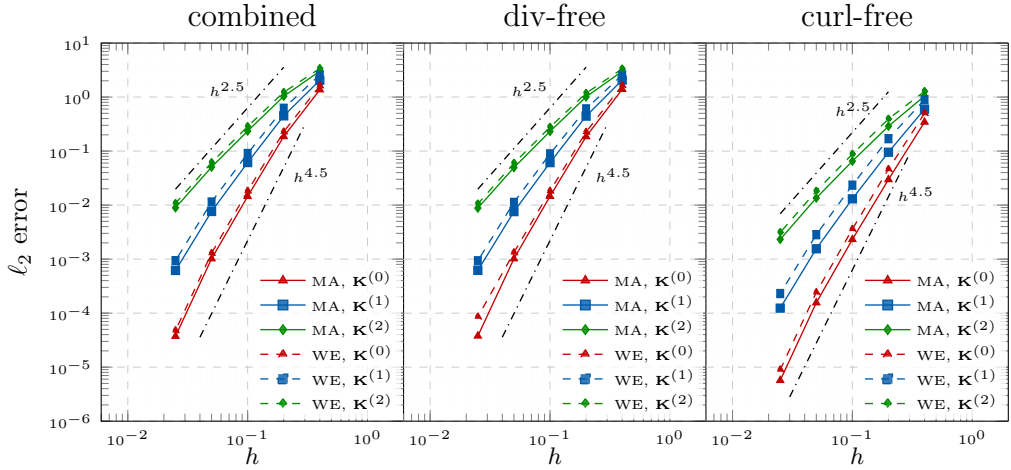

	\begin{figure}
		\centering
		\begin{tikzpicture}
			\begin{groupplot}[
				group style={
					group size=2 by 1, 
					horizontal sep=60pt, 
					vertical sep=0pt,
				},
				width=4.5cm, height=6cm, 
				scale only axis,           
				xmode=log, ymode=log,
				xmin=0.006, xmax=2,
				ymin=1e-15, ymax=100,
				grid=both,
				major grid style={line width=0.2pt, draw=gray!40, dashed},
				minor grid style={line width=0.1pt, draw=gray!15, dotted},
				minor x tick num=3, minor y tick num=3,
				ticklabel style={font=\tiny},
				xlabel={$h$},
				xlabel style={font=\small, yshift=6pt},
				legend style={
					at={(1,0)}, 
					anchor=south east,
					font=\tiny,
					cells={anchor=west},
					legend columns=1,
					inner sep=2pt,
					outer sep=1pt,
					draw=none,
					fill opacity=0.8,
					text opacity=1
				},
				legend image post style={mark size=1.8pt},
				cycle list={
					{color={myred}, mark=triangle*, line width=0.8pt,mark size=2.5pt},
					{color={myblue}, mark=square*, line width=0.8pt,mark size=1.5pt},
					{color={mygreen}, mark=diamond*, line width=0.8pt,mark size=2.5pt}
				},
				]
				
				\nextgroupplot[
				ylabel={\small Minimum eigenvalue}, 
				title={Uniform nodes},
				title style={font=\large, yshift=-3pt},
				]
				
				\addplot table[x index=0, y index=5, col sep=space] {field1divma.txt};
				\addlegendentry{$\mathbf{K}^{(0)}$}
				
				\addplot table[x index=0, y index=6, col sep=space] {field1divma.txt};
				\addlegendentry{$\mathbf{K}^{(1)}$}
                
				\addplot table[x index=0, y index=4, col sep=space] {field1divma.txt};
				\addlegendentry{$\mathbf{K}^{(2)}$}

                \addplot[
			domain=0.04:0.5,
			samples=2,
			color=black,
			line width=0.5pt,  %
			dashdotted,
			forget plot,
			] {0.001 * x^7};
			\node[anchor=south west, font=\tiny] at (axis cs:0.2,1e-9) {$q_X^{7}$};
			\addplot[
			domain=0.025:0.3,
			samples=2,
			color=black,
			line width=0.5pt,  %
			dashdotted,
			forget plot,
			] {30 * x^3};
			\node[anchor=south west, font=\tiny] at (axis cs:0.08,0.1) {$q_X^{3}$};

				\nextgroupplot[
                    title={Halton nodes},
                    title style={font=\large, yshift=-3pt},
                    ylabel={\small Minimum eigenvalue},
                ]

				\addplot table[x index=0, y index=5, col sep=space] {haltondivma_mineig.txt};
				\addlegendentry{$\mathbf{K}^{(0)}$}
				
				\addplot table[x index=0, y index=6, col sep=space] {haltondivma_mineig.txt};
				\addlegendentry{$\mathbf{K}^{(1)}$}
                
				\addplot table[x index=0, y index=4, col sep=space] {haltondivma_mineig.txt};
				\addlegendentry{$\mathbf{K}^{(2)}$}

                \addplot[
			domain=0.04:0.5,
			samples=2,
			color=black,
			line width=0.5pt,  %
			dashdotted,
			forget plot,
			] {0.001 * x^7};
			\node[anchor=south west, font=\tiny] at (axis cs:0.2,1e-9) {$q_X^{7}$};
			\addplot[
			domain=0.025:0.3,
			samples=2,
			color=black,
			line width=0.5pt,  %
			dashdotted,
			forget plot,
			] {30 * x^3};
			\node[anchor=south west, font=\tiny] at (axis cs:0.08,0.1) {$q_X^{3}$};

			\end{groupplot}
		\end{tikzpicture}
		
		\captionsetup{font=normalsize}
		\caption{Minimum eigenvalues of div-free interpolation matrices using Mat\'ern kernel $\phi_{7/2}$ on uniform and Halton nodes for three cases.}
		\label{fig.cond1}
	\end{figure}
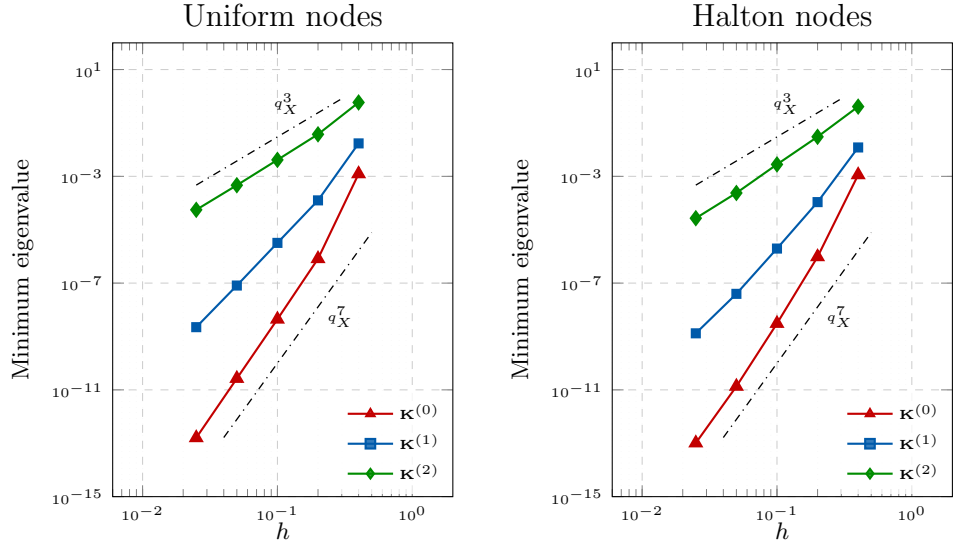

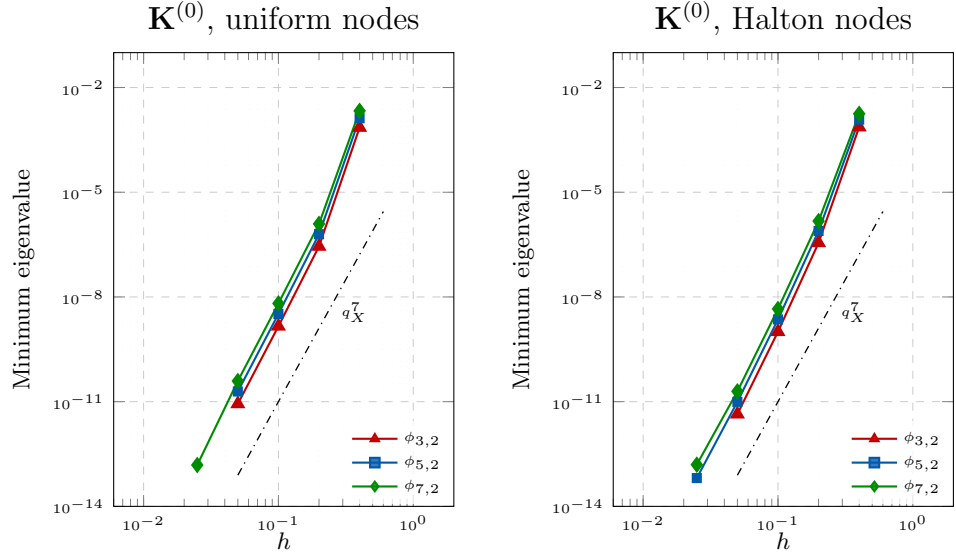
\begin{figure}
		\centering
		\begin{tikzpicture}
			\begin{groupplot}[
				group style={
					group size=2 by 1, 
					horizontal sep=60pt, 
					vertical sep=0pt,
				},
				width=4.5cm, height=6cm, 
				scale only axis,           
				xmode=log, ymode=log,
				xmin=0.006, xmax=2,
				ymin=1e-14, ymax=1e-1,
				grid=both,
				major grid style={line width=0.2pt, draw=gray!40, dashed},
				minor grid style={line width=0.1pt, draw=gray!15, dotted},
				minor x tick num=3, minor y tick num=3,
				ticklabel style={font=\tiny},
				xlabel={$h$},
				xlabel style={font=\small, yshift=6pt},
				legend style={
					at={(1,0)}, 
					anchor=south east,
					font=\tiny,
					cells={anchor=west},
					legend columns=1,
					inner sep=2pt,
					outer sep=1pt,
					draw=none,
					fill opacity=0.8,
					text opacity=1
				},
				legend image post style={mark size=1.8pt},
				cycle list={
					{color={myred}, mark=triangle*, line width=0.8pt,mark size=2.5pt},
					{color={myblue}, mark=square*, line width=0.8pt,mark size=1.5pt},
					{color={mygreen}, mark=diamond*, line width=0.8pt,mark size=2.5pt}
				},
				]
				
				\nextgroupplot[
				ylabel={\small Minimum eigenvalue}, 
				title={$\mathbf{K}^{(0)}$, uniform nodes},
				title style={font=\large, yshift=-3pt},
				]
				
				\addplot table[x index=0, y index=1, col sep=space] {diffdivWEmineig.txt};
				\addlegendentry{$\phi_{3,2}$}
				
				\addplot table[x index=0, y index=2, col sep=space] {diffdivWEmineig.txt};
				\addlegendentry{$\phi_{5,2}$}
                
				\addplot table[x index=0, y index=3, col sep=space] {diffdivWEmineig.txt};
				\addlegendentry{$\phi_{7,2}$}

			\addplot[
			domain=0.05:0.6,
			samples=2,
			color=black,
			line width=0.5pt,  %
			dashdotted,
			forget plot,
			] {0.0001 * x^7};
			\node[anchor=south west, font=\tiny] at (axis cs:0.25,1e-9) {$q_X^{7}$};

				\nextgroupplot[
                    title={ $\mathbf{K}^{(0)}$, Halton nodes},
                    title style={font=\large, yshift=-3pt},
                    ylabel={\small Minimum eigenvalue},
                ]

				\addplot table[x index=0, y index=4, col sep=space] {diffdivWEmineig.txt};
				\addlegendentry{$\phi_{3,2}$}
				
				\addplot table[x index=0, y index=5, col sep=space] {diffdivWEmineig.txt};
				\addlegendentry{$\phi_{5,2}$}
                
				\addplot table[x index=0, y index=6, col sep=space] {diffdivWEmineig.txt};
				\addlegendentry{$\phi_{7,2}$}

                \addplot[
			domain=0.05:0.6,
			samples=2,
			color=black,
			line width=0.5pt,  %
			dashdotted,
			forget plot,
			] {0.0001 * x^7};
			\node[anchor=south west, font=\tiny] at (axis cs:0.25,1e-9) {$q_X^{7}$};
			
			\end{groupplot}
		\end{tikzpicture}
		
		\captionsetup{font=normalsize}
		\caption{Minimum eigenvalues of interpolation matrices for the div-free component $\boldsymbol{u}_1$  using Wendland kernels $\phi_{3,2}$, $\phi_{5,2}$, and $\phi_{7,2}$ with the same shape parameter $\epsilon = 0.37$ on uniform nodes and Halton nodes. }
		\label{fig.mineig}
	\end{figure}

	\subsection{Simulation of a div-free field}
	
	To further evaluate the proposed methods on a flow with nontrivial topology, we consider a div-free target field defined on the square domain $\Omega=[-1.5, 1.5]^2 $. The field is generated by a stream function $\psi (\boldsymbol {x})$ characterized by multiple local extrema:
\begin{equation}\label{eq:div_free_field}
        \psi(\boldsymbol{x}) = -2g\left(\frac{27}{2}\|\boldsymbol{x}\|^4\right) - \frac{1}{2}g(27\|\boldsymbol{x}\|^2) - 2\sum_{j=0}^4 g(9\|\boldsymbol{x}-\boldsymbol{\xi}_j\|^2),
    \end{equation}
	with $g(r) = \exp(r)/(1+\exp(r))^2$ and $$\boldsymbol{\xi}_j = \big(\cos(2\pi j/5 + 0.1), \sin(2\pi j/5 + 0.1)\big)^\top.$$
	The associated div-free velocity field is obtained via the standard planar rotation of the gradient, namely $\boldsymbol{u}(\boldsymbol{x}) = (-\partial_{x_2} \psi, \partial_{x_1} \psi)^\top$.
	
	We discretize $\Omega$ using a uniform set of nodes $X \subset \Omega$ with fill distance $h_X = 0.04$ and we use the Mat\'{e}rn kernel of order $\nu=7/2$. \cref{fig1}(a) shows contour lines of $\psi$, and \cref{fig1}(b) displays streamlines of the target field $\boldsymbol{u}$. \cref{fig1}(c) depicts the field reconstructed by using $\mathbf{K}^{(0)}$. The reconstruction captures the main topological features of the flow with high fidelity.
    \cref{fig2} compares the pointwise approximation errors for three cases using a common scale. Overall, $\mathbf{K}^{(0)}$ yields the smallest errors, with the most pronounced improvements in regions where $\bld{u}$ varies rapidly; in these areas, $\mathbf{K}^{(0)}$ better controls the localized approximation error.
	
	\begin{figure}
		\centering
		\begin{subfigure}{0.32\textwidth}
			\centering
			\includegraphics[width=\linewidth]{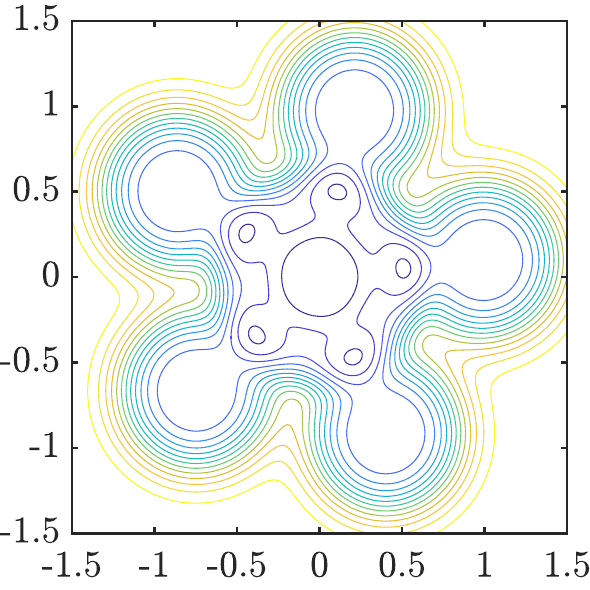}
			\caption{Contour of $\psi$}
			\label{fig:psi_contour}
		\end{subfigure}
		\begin{subfigure}{0.32\textwidth}
			\centering
			\includegraphics[width=\linewidth]{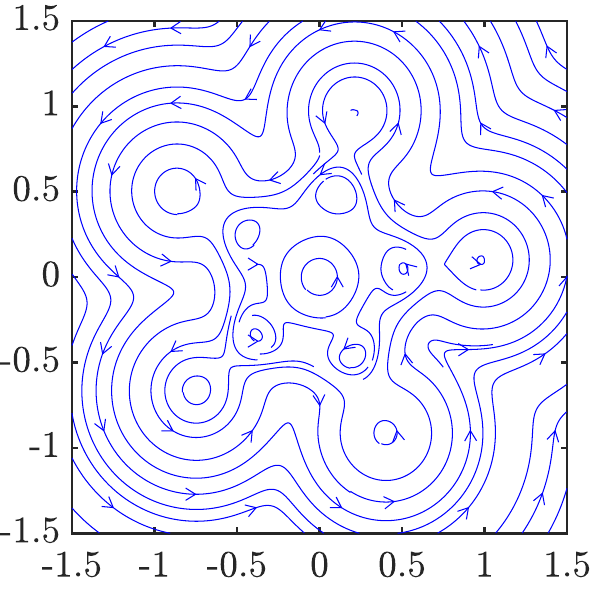}
			\caption{Streamlines of $\boldsymbol{u}$}
			\label{Figure:streamlines}
		\end{subfigure}
		\begin{subfigure}{0.32\textwidth}
			\centering
			\includegraphics[width=\linewidth]{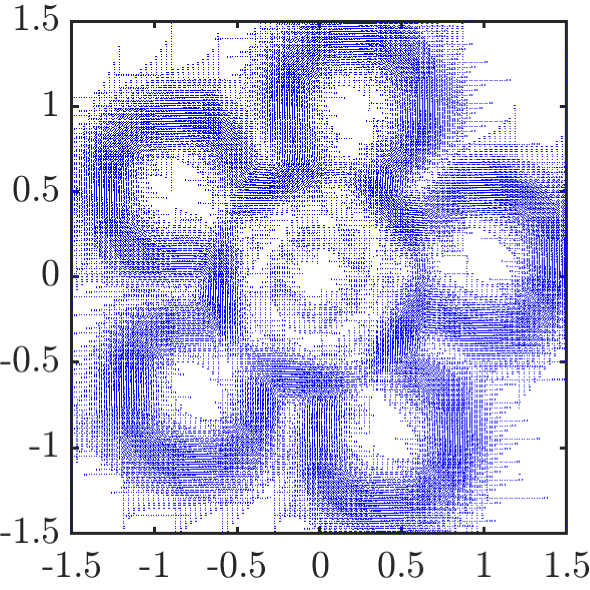}
			\caption{Reconstructed vector field}
			\label{fig:interpolated}
		\end{subfigure}
		\caption{Contour of the potential $\psi$, streamlines of the induced vector field $\boldsymbol{u}$, and the vector field reconstructed using $\mathbf{K}^{(0)}$.}
		\label{fig1}
	\end{figure}

	\begin{figure}
		\centering
        \makebox[0pt][r]{%
    \hspace{-1cm}
    \raisebox{2.7cm}{\rotatebox{90}{\small Error}}
    }
		\begin{subfigure}{0.28\textwidth}
			\centering
			\includegraphics[width=\linewidth]{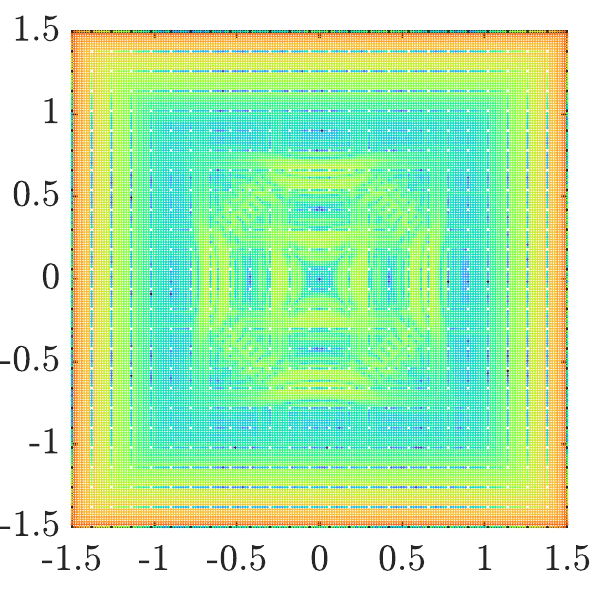}
			\caption{ $\mathbf{K}^{(0)}$}
		\end{subfigure}
        \hspace{0.005cm}
		\begin{subfigure}{0.28\textwidth}
			\centering
			\includegraphics[width=\linewidth]{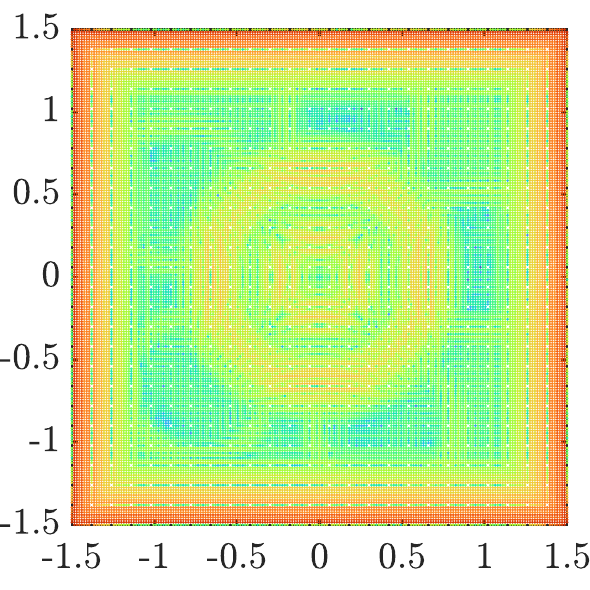}
			\caption{ $\mathbf{K}^{(1)}$}
		\end{subfigure}
		\raisebox{0.1cm}{
    \begin{subfigure}{0.365\textwidth}
	   \centering
	   \includegraphics[width=\linewidth]{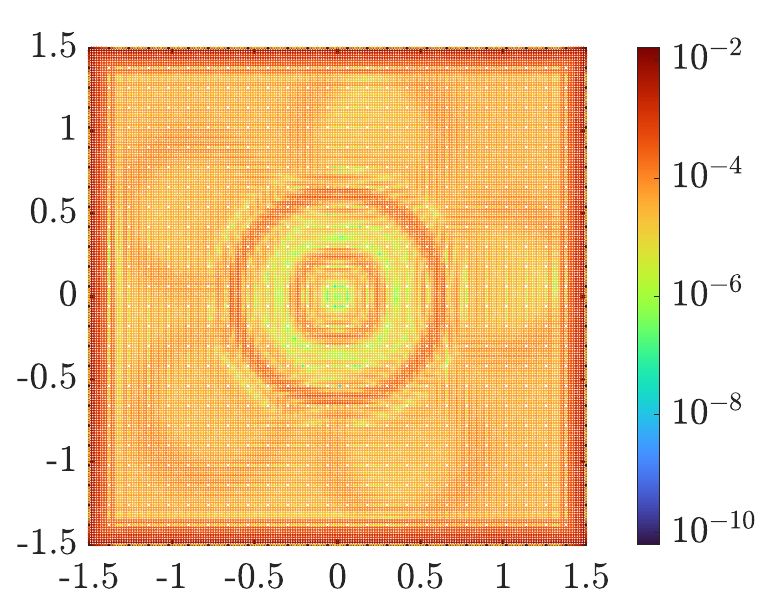}
       \vspace{-0.35cm}
	   \caption{$\mathbf{K}^{(2)}$}
    \end{subfigure}
    }
		\caption{Pointwise approximation errors for three cases in reconstructing the div-free field generated by the stream function $\psi$ given in \eqref{eq:div_free_field}.}
		\label{fig2}
	\end{figure}
	
	\subsection{Simulation on an annular domain}
	
	Finally, we consider the annulus $$\Omega = \{\boldsymbol{x} \in \mathbb{R}^2 : 0.75 \le \|\boldsymbol{x}\| \le 2\},$$
	and constructed a mixed vector field $\mathbf{f}$ by superposing a div-free component generated by the stream function $\psi(\boldsymbol{x}) = \cos(2\|\boldsymbol{x}\|^2)$ with a curl-free component given by the gradient of \textit{peaks} function $p(\boldsymbol{x})$. Specifically, 
\begin{equation}\label{eq:vec_field_annulus}
        \mathbf{f} = \begin{pmatrix} -\partial_{x_2} \psi \\ \partial_{x_1} \psi \end{pmatrix} + \nabla p.
    \end{equation}
	
	To approximate $\mathbf{f}$ and to recover its div-free and curl-free parts, we use a Mat\'ern kernel with smoothness $\nu=5$ and shape parameter $\varepsilon=4$, and we compare three cases. The discretization uses $N=1246$ interpolation centers and $M=2794$ evaluation points, both distributed uniformly in $\Omega$. \cref{fig3} visualizes the full field $\mathbf{f}$ via streamlines and shows contour plots of the scalar potentials $\psi$ and $p$ associated with its div-free and curl-free components, respectively. \cref{fig4} presents the reconstructed vector fields. All three cases capture the dominant flow patterns across the annulus, including near the curved boundaries. \cref{fig5} reports pointwise errors for the full field; the div-free and curl-free components shows the same qualitative behavior. In all cases, $\mathbf{K}^{(0)}$ delivers the smallest errors among the three approaches, which further verifies the theoretical results.

	\begin{figure}
		\centering
		\begin{subfigure}{0.32\textwidth}
			\centering
			\includegraphics[width=\linewidth]{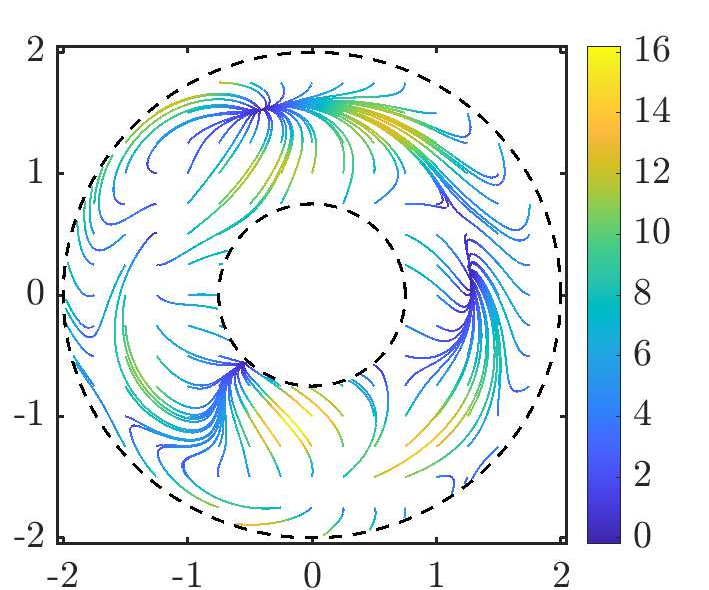}
			\caption{Streamlines of $\mathbf{f}$}
		\end{subfigure}
		\hfill
		\begin{subfigure}{0.32\textwidth}
			\centering
			\includegraphics[width=\linewidth]{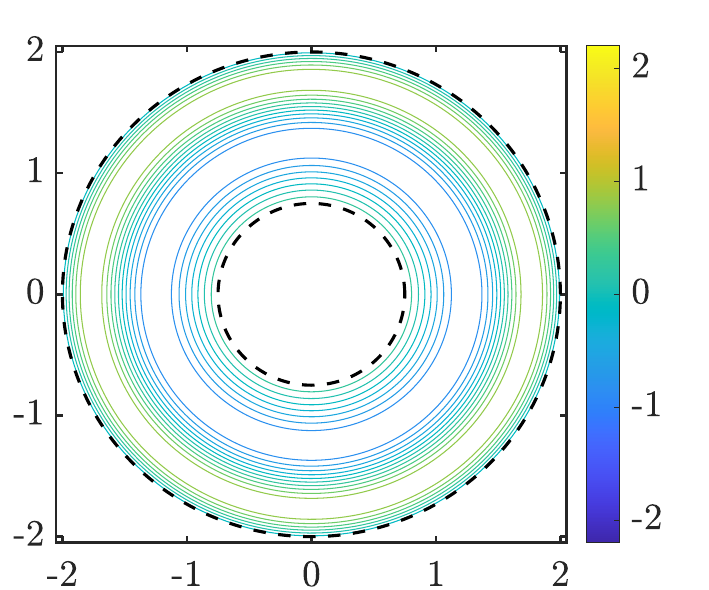}
			\caption{Contour of $\psi$}
		\end{subfigure}
		\hfill
		\begin{subfigure}{0.32\textwidth}
			\centering
			\includegraphics[width=\linewidth]{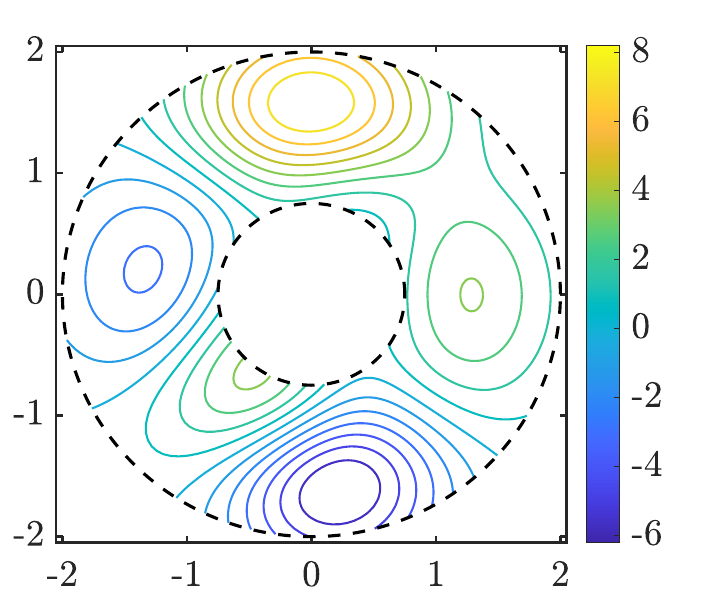}
			\caption{Contour of $p$}
		\end{subfigure}
		\caption{Target field \eqref{eq:vec_field_annulus} in the annulus: (a) streamlines of the field; (b)–(c) scalar potentials of its decomposed components.}
		\label{fig3}
	\end{figure}

	\begin{figure}
		\centering
		\begin{subfigure}{0.32\textwidth}
			\centering
			\includegraphics[width=\linewidth]{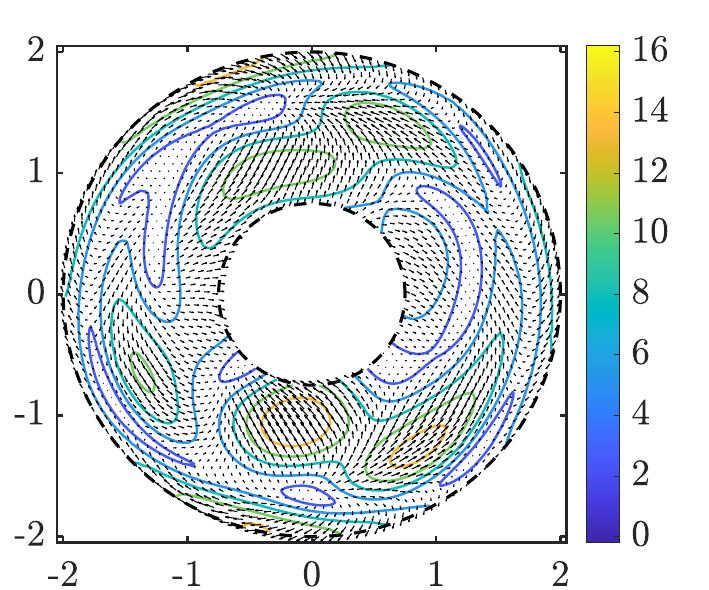}
			\caption{Interpolated $\mathbf{f}$}
		\end{subfigure}
        \hspace{0.02cm}
		\begin{subfigure}{0.32\textwidth}
			\centering
			\includegraphics[width=\linewidth]{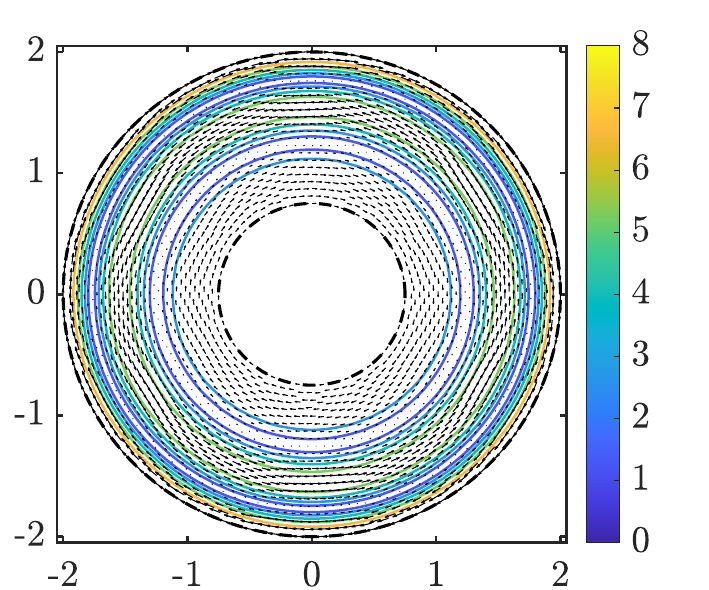}
			\caption{Div-free component}
		\end{subfigure}
        \hspace{0.08cm}
		\begin{subfigure}{0.32\textwidth}
			\centering
			\includegraphics[width=\linewidth]{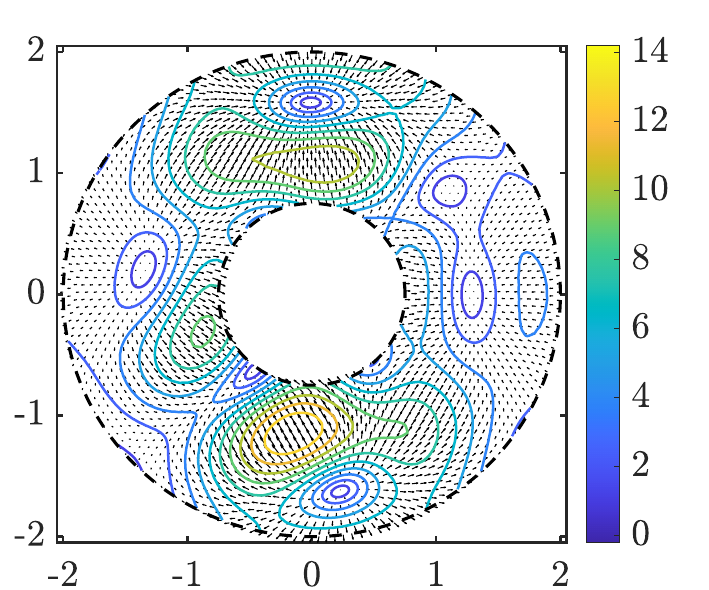}
			\caption{Curl-free component}
		\end{subfigure}
		\caption{Contour plots for matrix-valued kernel interpolation of $\mathbf{f}$ and its two components using $\mathbf{K}^{(0)}$.}
		\label{fig4}
	\end{figure}

	\begin{figure}
		\centering
        \makebox[0pt][r]{%
    \hspace{-1.2cm}
    \raisebox{2.7cm}{\rotatebox{90}{\small Error}}
    }
        \hspace{-0.2cm}
		\begin{subfigure}{0.28\textwidth}
			\centering
			\includegraphics[width=\linewidth]{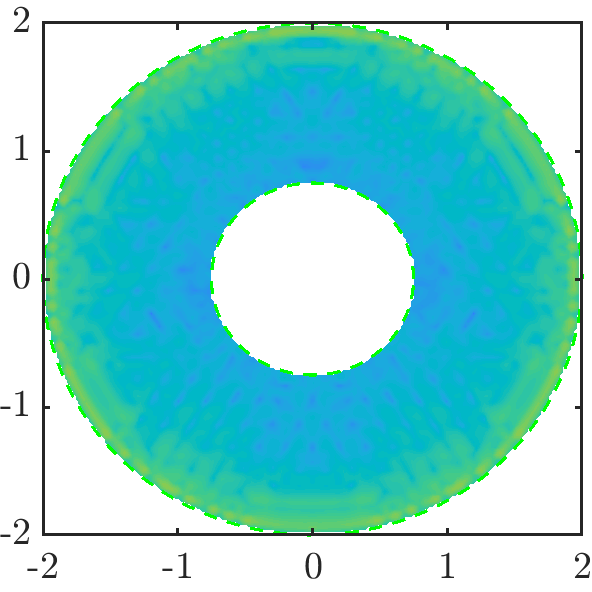}
			\caption{$\mathbf{K}^{(0)}$}
		\end{subfigure}
        \hspace{0.26cm}
		\begin{subfigure}{0.28\textwidth}
			\centering
			\includegraphics[width=\linewidth]{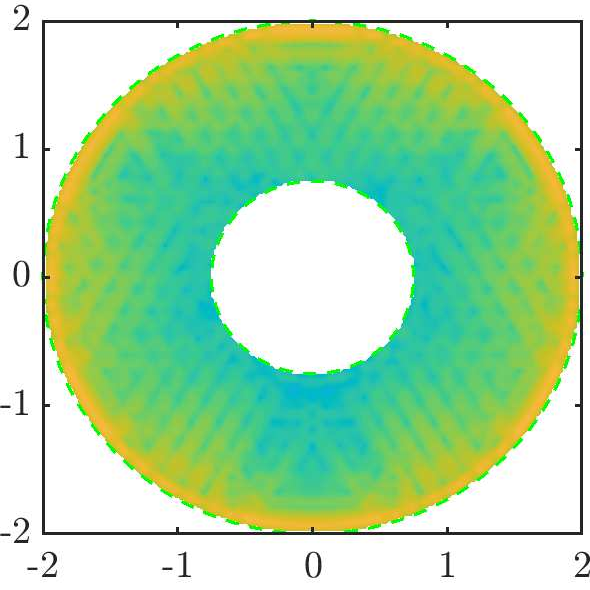}
			\caption{$\mathbf{K}^{(1)}$}
		\end{subfigure}
        \hspace{0.24cm}
        \raisebox{0.05cm}{
		\begin{subfigure}{0.336\textwidth}
			\centering
			\includegraphics[width=\linewidth]{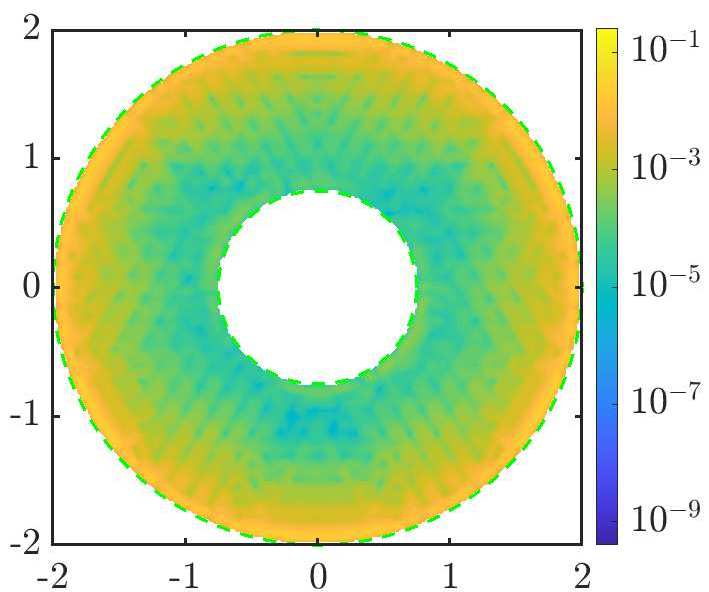}
			\caption{$\mathbf{K}^{(2)}$}
		\end{subfigure}}
		\caption{Pointwise approximation errors for three cases in approximating the combined field  \eqref{eq:vec_field_annulus} in the annulus.}
		\label{fig5}
	\end{figure}


\bibliographystyle{plain}
\bibliography{reference}

\end{document}